\documentclass{amsart}
\usepackage{amssymb}
\usepackage{amsfonts}
\usepackage{amsmath,amsthm}
\usepackage{natbib}

\newtheorem{theorem}{Theorem}[section]
\newtheorem{corollary}[theorem]{Corollary}
\newtheorem{lemma}[theorem]{Lemma}
\newtheorem{proposition}[theorem]{Proposition}
\newtheorem{definition}[theorem]{Definition}
\newtheorem{remark}[theorem]{Remark}
\newtheorem{assumption}{Assumption}

\newcommand{\EE}{\ensuremath{\mathbb{E}}}
\newcommand{\Prob}{\ensuremath{\mathbb{P}}}

\newcommand{\ul}{\ensuremath{\lfloor t/\Delta_n\rfloor}}
\newcommand{\uls}{\ensuremath{\lfloor s/\Delta_n\rfloor}}

\newcommand{\R}{\ensuremath{\mathbb{R}}}
\newcommand{\N}{\ensuremath{\mathbb{N}}}
\newcommand{\SRV}{SARCV}

\allowdisplaybreaks

\begin{document}

%\title[Law of large numbers and volatility in Hilbert space]{Weak law of large numbers and volatility estimation in a Hilbert space setting}

\title[Law of large numbers for realised covariation in Hilbert space]{A weak law of large numbers for realised covariation  in a Hilbert space setting}

\author{Fred Espen Benth \and Dennis Schroers \and Almut E. D. Veraart}
\date{\today}

\thanks{F. E. Benth and A. E. D. Veraart would like to thank the Isaac Newton Institute for Mathematical Sciences for support and hospitality during the programme  \emph{The Mathematics of Energy Systems} when parts of the work on this paper were undertaken. This work was supported by:
EPSRC grant number EP/R014604/1.
D. Schroers and F. E. Benth gratefully acknowledge financial support from the STORM project 274410, funded by the Research Council of Norway, and the thematic research group SPATUS, funded by UiO:Energy at the University of Oslo.}

\address{Fred Espen Benth: Department of Mathematics, University of Oslo, P.O. Box 1053, Blindern, 0316, OSLO, Norway}

\address{Dennis Schroers: Department of Mathematics, University of Oslo, P.O. Box 1053, Blindern, 0316, OSLO, Norway}

\address{Almut E. D. Veraart: Department of Mathematics, Imperial College London, 180 Queen’s Gate, London, SW7 2AZ, UK}

\begin{abstract}
  % In a framework of functional high-frequency estimation, we prove u.c.p. convergence with respect to the Hilbert-Schmidt norm of a 
   %   natural generalisation of the realised covariation for the trace-class operator-valued integrated volatility process corresponding to general mild solutions of Hilbert space-valued stochastic evolution equations in the sense of \cite{DPZ2014}.

This article generalises the concept of realised covariation to Hilbert-space-valued stochastic processes. More precisely, based on high-frequency functional data, we construct an estimator of the trace-class operator-valued integrated volatility process arising in general mild solutions of Hilbert space-valued stochastic evolution equations in the sense of \cite{DPZ2014}. 
We prove a weak law of large numbers for this estimator, where the convergence is uniform on compacts in probability with respect to the Hilbert-Schmidt norm. 
In addition, we show that the conditions on the volatility process are valid for most common stochastic volatility models in Hilbert spaces. 
\end{abstract}


%\allowdisplaybreaks
%\begin{document}

\maketitle
%\tableofcontents
%\newpage
%\Large
%Aim: Volatility estimation in infinite dimensions:  Weak law of large numbers for operator realised covariation process.

%\begin{abstract}
%In a framework of functional high-frequency estimation, We prove u.c.p. convergence with respect to the Hilbert-Schmidt norm of a natural generalisation of the realised covariation for the Trace-class operator valued integrated volatility process corresponding to general mild solutions of Hilbert space-valued stochastic evolution equations in the sense of \cite{DPZ2014}.
%In addition, we show that the conditions on the volatility process are valid for most stochastic volatility models in Hilbert spaces. 
%\end{abstract}

%\newpage
\section{Introduction }
Stochastic volatility and covariance estimation are of key importance in many fields. Motivated in particular by financial applications, a lot of research has been devoted to constructing suitable (co-) volatility estimators and to deriving their asymptotic limit theory in the setting when discrete, high-frequent observations are available. Initially, the main interest was in (continuous-time) stochastic models based on (It\^{o}) semimartingales, where the so-called realised variance and covariance estimators (and their extensions) proved to be powerful tools. Relevant articles include the works by \cite{BNS2002, BNS2003, BNS2004, ABDL2003} and \cite{Jacod2008}, amongst many others,
and the textbooks by  \cite{JacodProtter2012} and  \cite{Ait-SahaliaJacod2014}.

Subsequently, the theory was extended to cover non-semimartingale models, see, for instance, \cite{CNW2006}, \cite{BNCP11}, \cite{BNCP10b}, \cite{CorcueraHPP2013},  \cite{CorcueraNualartPodolskij2015}  and the survey by \cite{Podolskij2014}, where the proofs of the asymptotic theory rely on Malliavin calculus and the famous fourth-moment theorem, see \cite{NP2005}. The multivariate theory has been studied in \cite{VG2019, PV2019}.

Common to these earlier lines of investigation is the fact that the stochastic processes considered have finite dimensions.
In this article, we extend the concept of realised covariation to an infinite-dimensional framework. 

The estimation of covariance operators is elementary in the field of functional data analysis and was elaborated mainly for discrete-time series of functional data (see e.g. \cite{Ramsay2005}, \cite{Ferraty2006}, \cite{Yao2005}, \cite{Bosq2012}, \cite{Horvath2012}, \cite{Panaretos2013}).
 However, spatio-temporal data that can be considered as functional might also be sampled densely in time, like forward curves for interest rates or commodities and data from  geophysical and environmental applications.

In this paper, we consider a 
separable Hilbert space $H$ and 
study 
 $H$-valued stochastic processes $Y$ of the form
\begin{equation}\label{Intro Mild SPDE}
Y_t=\mathcal{S}(t)h+\int_0^t \mathcal{S}(t-s)\alpha_sds+\int_0^t\mathcal S(t-s)\sigma_s dW_s,\quad t\in[0,T],
\end{equation}
for some $T>0$. 
Here $(\mathcal S (t))_{t\geq 0}$ is a strongly continuous semigroup, $\alpha:=(\alpha_t)_{t\in[0,T]}$ a predictable and almost surely integrable $H$-valued stochastic process, $\sigma:=(\sigma_t)_{t\in [0,T]}$ is a predictable operator-valued process, $h\in H$ some initial condition and $W$ a so called $Q$-Wiener process on $H$ (see Section \ref{sec: Preliminaries} below for details).

Our aim is to construct an estimator 
 for the integrated covariance process 
 \begin{equation*}
 \left(\int_0^t\sigma_sQ\sigma_s^*ds\right)_{t\in[0,T]}.    
 \end{equation*}
 More precisely, we denote by 
  \begin{equation}\label{Intro Realised Volatility}
 \sum_{i=1}^{\ul}(Y_{t_i}-\mathcal S(\Delta_n)Y_{t_{i-1}})^{\otimes 2},
\end{equation}
the \emph{
semigroup-adjusted realised covariation (\SRV)}
for an equally spaced grid $t_i:=i\Delta_n$ for $\Delta_n=1/n$, $i=1,\dots,\lfloor t/\Delta_n\rfloor$. 
We  prove uniform convergence in probability (ucp) with respect to the Hilbert-Schmidt norm of the \SRV\ 
to the integrated covariance process under mild conditions on the volatility.

This framework differs from common high-frequency settings mainly due to peculiarities that arise from infinite dimensions. First, observe  that the main motivation to consider processes in this form, is that a vast amount of parabolic stochastic partial differential equations
posses only mild (in opposition to analytically strong) solutions, which are of the form \eqref{Intro Mild SPDE}. That is, $Y$ is (under weak conditions) the mild solution of a stochastic partial differential equation
\begin{align*}
(\text{SPDE}) \quad dX_t= (AX_t +\alpha_t) dt+ \sigma_t dW_t, \quad X_0=h,\quad t\in [0,T].
\end{align*}
 (cf. \cite{DPZ2014}, \cite{PZ2007} or \cite{GM2011}).

In contrast to finite-dimensional stochastic diffusions, this is a priori not an $H$-valued semimartingale, but rather an $H$-valued Volterra process. 

Various recent developments related to statistical inference  for (parabolic) SPDEs based on discrete observations in time and space have emerged, see e.g.~\cite{Cialenco2020}, \cite{Bibinger2020}, \cite{Chong2020}, \cite{ChongDalang2020}.

To 
the best of our knowledge, our paper is the first one considering high-frequency estimation of (co-) volatility  of infinite-dimensional stochastic evolution equations in an operator setting.
This is of interest  for various reasons. For instance, a simple and important application might be the parameter estimation for $H$-valued Ornstein-Uhlenbeck process (that is, $\sigma_s=\sigma$ is a constant operator).
Elementary techniques such as functional principal component analysis might then be considered on the level of volatility. In a multivariate setting, dynamical dimension reduction was conducted for instance in \cite{AIT-Sahalia2019}.
Furthermore, it can be used as a tool for inference of infinite-dimensional stochastic volatility models as in \cite{BenthRudigerSuss2018} or \cite{BenthSimonsen2018}.
In the special case of a semigroup that is continuous with respect to the operator norm, the framework also covers the estimation of volatility for $H$-valued semimartingales.

%Let us further remark, that our work is different from the branch spatiotemporal covariance estimation like in PANARETOSMASAK,ETC not only in that it estimates volatility rather than covariance, but also in the sense that the estimation is based on one spatio-temporal observation, rather than a collection of i.i.d. spatiotemporal observations.

We organize the paper as follows: 
%After a comparison with existing literature on spatio-temporal covariance estimation and and volatility estimation for parabolic stochastic partial differential equations, 
First, we recall the main technical preliminaries of our framework in Section \ref{sec: Preliminaries}. In Section \ref{sec: Weak Law of large numbers}, we establish the weak law of large numbers. For that,  we discuss the conditions  imposed on the volatility process in Section \ref{sec: Technical assumptions} and state our main result, given by Theorem \ref{T: LLN for semigroup case}, in Section \ref{sec: Main result}. Afterwards, we show how to weaken the assumptions on the volatility by a localization argument in Section \ref{sec: Extension by localization}. In Section \ref{sec: Applications}, we study the behaviour of the estimator in special cases of semigroups and volatility.
We discuss conditions for particular examples of semigroups to determine the speed of convergence of the estimator in Section \ref{sec: Convergence rates and semigroups}. In Section \ref{sec: Stochastic volatility models}, we validate our assumptions for some stochastic volatility models in Hilbert spaces.  Section 5 is devoted to the proofs of our main results, while in  %\ref{Sec: Connection to data analysis}, we  comment on the task of additional spatial approximation procedures in the case that $H$ is a space of functions. Finally, 
Section \ref{Sec:Conclusion}  we discuss our results and methods in relation to some existing literature and provide some outlook into further developments. 
Some technical proofs are relegated to the Appendix.

\section{Notation and some preliminary results}\label{sec: Preliminaries}
Let
$(\Omega, \mathcal{F}, (\mathcal{F}_t)_{t\geq 0}), \Prob)$ denote a filtered probability space satisfying the usual conditions.
Consider two separable Hilbert spaces $U, H$ with scalar products denoted by $\langle \cdot, \cdot \rangle_U$,  $\langle \cdot, \cdot \rangle_H$  and norms $\Vert\cdot\Vert_U$, $\Vert\cdot\Vert_H$, respectively. We denote $L(U,H)$ the space of all linear bounded operators $K:U\rightarrow H$, and use the shorthand notation $L(U)$ for $L(U,U)$. Equipped with the operator norm, $L(U,H)$ becomes a Banach space. The adjoint operator of a $K\in L(U,H)$ is denoted by $K
^*$, and is an element on $L(H,U)$.  
%
%\begin{definition}
%    A linear operator $T:U\to H$ is \emph{bounded}, if there exists a constant $C>0$ such that
%    \begin{align*}
%        \Vert Tu\Vert_H \leq C \Vert u\Vert_U, \text{ for all } u \in U.
%    \end{align*}
%\end{definition}
%By linearity, we see that any bounded linear operator $T$ is also continuous (in fact Lipschitz continuous). 
%Note that throughout the paper, we will often work with \emph{random} operators. In this case, boundedness refers to the existence of a nonnegative random variable $C$, such that
%\begin{align*}
%    \Prob\{\omega\in \Omega: \Vert Tu(\omega)\Vert_H \leq C(\omega) \Vert u\Vert_U,  \text{ for all } u \in U\}=1.
%\end{align*}
%{\bf We probably also want to assume that $E(C)<\infty$ or $\mathbb P(C<\infty)=1$. Check if this is something we really need(want/use?}

Following \citet[Appendix A]{PZ2007} we use the following notations:
%\begin{definition}
%Let $L(U,H)$ denote the Banach space of  linear continuous (i.e.~bounded) operators from $U$ to $H$. 
%\begin{enumerate}
%\item 
%We define the corresponding \emph{operator norm}
%as
%\begin{align*}
%    \Vert R\Vert_{\text{op}}:=\sup_{x\in U, x \not =0}\frac{\Vert Rx\Vert_H}{\Vert x\Vert_U}, \text{ for } R \in L(U,H).
%\end{align*}
%\item For an $R\in L(U,H)$ we define the \emph{adjoint operator} $R^*\in L(H,U)$ which is uniquely defined by $\langle R^*h,u\rangle_U=\langle h, Ru\rangle_H$ for all $h\in H, u\in U$.
An operator $K\in L(U,H)$ is called \emph{nuclear} or \emph{trace class} if the following representation holds
\begin{align*}
    K u = \sum_k b_k\langle u, a_k \rangle_U, \text{ for } u \in U,
\end{align*}
where $\{a_k\} \subset U$ and $\{b_k\}\subset H$ such that $\sum_k\Vert a_k\Vert_U\Vert b_k\Vert_H<\infty$. The  space of all nuclear operators is denoted by $L_1(U,H)$; it is a separable Banach space and its norm is denoted by
\begin{align*}
    \Vert K\Vert_1 :=\inf\left\{\sum_k\Vert a_k\Vert_U\Vert b_k\Vert_H: K u = \sum_k b_k\langle u, a_k \rangle_U\right\}.
\end{align*}
We denote by $L^+_1(U,H)$ the class of all symmetric, non-negative-definite nuclear operators from  $U$ to $H$.
We write $L_1(U)$ and
$L_1^+(U)$ for $L_1(U,U)$ and $L_1^+(U,U)$, resp. Frequently, nuclear operators are also called trace class operators.

For $x\in U$ and $y\in H$, we define the tensor product $x\otimes y$ as the linear operator in $L(U,H)$ defined as $x\otimes y(z):=\langle x, z\rangle_U y$ for $z\in U$. We note that $x \otimes y \in L_1(U,H)$ and $\Vert x \otimes y \Vert_1 =\Vert x\Vert_U \Vert y\Vert_H$, see \citet[p.~107]{PZ2007}.
%{\bf Check this result!!}

The operator $K\in L(U,H)$ is said to be a \emph{Hilbert-Schmidt operator} if
\begin{align*}
    \sum_k \Vert K e_k\Vert_H^2 < \infty,
\end{align*}
for any orthonormal basis (ONB) $(e_k)_{k\in\mathbb N}$ of $U$. The space of all Hilbert-Schmidt operators is denoted by $L_{\text{HS}}(U,H)$. We can introduce an inner product by
%, which is a separable Hilbert space with inner product given by
\begin{align*}
    \langle K, L \rangle_{\text{HS}}:=\sum_{k}\langle Ke_k, Le_k\rangle_H, \text{ for }  K,L \in L_{\text{HS}}(U,H).
\end{align*}
The induced norm is denoted $\Vert\cdot\Vert_{\text{HS}}$. %and $L_{\text{HS}}(U,H)$ is again a separable Hilbert space. 
As usual, we write $L_{\text{HS}}(U)$ in the case $L_{\text{HS}}(U,U)$.

%\end{enumerate}
%\end{definition}
%If $H=U$, we denote $L(U):=L(U,U)$ and so on. 

%\subsection{A useful inequality for Hilbert-Schmidt operators}
We have the following convenient result for the space of Hilbert-Schmidt operators. Although it is well-known, we include the proof of this result in the Appendix \ref{App:proofs} for the convenience of the reader:
\begin{lemma}
\label{lem:HS-banachalg}
Let $U,V,H$ be separable Hilbert spaces. Then $L_{\text{HS}}(U,H)$ is a separable 
Hilbert space. Moreover, if $K\in L_{\text{HS}}(U,V), L\in L_{\text{HS}}(V,H)$, then
$LK\in L_{\text{HS}}(U,H)$ and 
\begin{equation}
\Vert LK\Vert_{\text{HS}}\leq \Vert L \Vert_{\text{op}}\Vert K\Vert_{\text{HS}}\leq \Vert L \Vert_{\text{HS}}\Vert K\Vert_{\text{HS}},
\end{equation}
where the HS-norms are for the spaces in question.
\end{lemma}

%{\bf Do we need the following notation?}We denote by $\mathcal{M}^2(U)$ the space of all square integrable martingales in $U$ with respect to $(\mathcal{F}_t)_{t\geq 0}$.

\subsection{Hilbert-space-valued stochastic integrals}
Fix $T>0$ and assume that $0\leq t \leq T$ throughout.
Let $W$ denote a Wiener process taking values in $U$ with covariance operator $Q\in L^+_1(U)$.
\begin{definition}
A stochastic process $(W_t)_{t\geq 0}$ with values in $U$ is called Wiener process with covariance operator $Q \in L_1^+(U)$, if $W_0=0$ almost surely, $W$ has independent and stationary increments, and for $0\leq s \leq t$, we have $W_t-W_s\sim N(0,(t-s)Q)$.
\end{definition}
\begin{remark}
Recall that a $U$-valued random variable $X$ is normal with mean $a\in U$ and covariance operator $Q\in L_1^+(U)$ if 
$\langle X,f\rangle_U$ is a real-valued normally distributed random variable for each $f\in U$, with mean $\langle a,f\rangle$ and
$$
E[\langle X,f\rangle_U\langle X,g\rangle_U]=\langle Qf,g\rangle_U, \forall f,g\in U.
$$
\end{remark}

We introduce the space $\mathcal L_{2,T}(U,H)$ of predictable $L(U,H)$-valued stochastic processes $Z=(Z_t)_{t\geq 0}$
such that 
\begin{equation}
    \mathbb E\left[\int_0^T\Vert Z_sQ^{1/2}\Vert_{\text{HS}}^2ds\right]<\infty,
\end{equation}
for $T<\infty$. 
Then $\mathcal L_{2,T}(U,H)$ will be the space of integrable processes with respect to the $Q$-Wiener
process $W$ on $[0,T]$.

Let $\sigma=(\sigma_t)_{t\geq 0}$ denote a stochastic volatility process 
where $\sigma_t\in \mathcal L_{2,T}(U,H)$ for some fixed $T<\infty$.
 The stochastic integral
\begin{align*}
    Y_t:= \int_0^t \sigma_s dW_s
\end{align*}
can then be defined as in \cite[Chapter 8]{PZ2007} and takes values in the Hilbert space $H$.

We denote the tensor product of the stochastic integral $Y$ by 
$   \left(Y_t\right)^{\otimes 2}=Y_t\otimes Y_t$,
and define the corresponding stochastic variance term as the \emph{operator angle bracket} (not to be confused with the inner products introduced above!) given by 
\begin{align*}
    \langle\langle Y\rangle\rangle_t=\int_0^t\sigma_sQ\sigma^*_sds=\int_0^t
    (\sigma_sQ^{1/2}) (\sigma_sQ^{1/2})^*ds,
\end{align*}
see \citet[Theorem 8.7, p.~114]{PZ2007}.
\begin{remark}
As in \citet[p.~104]{DPZ2014}, we note that $(\sigma_sQ^{1/2})\in L_{HS}(U,H)$ and $(\sigma_sQ^{1/2})^*\in L_{HS}(H,U)$. Hence the process $(\sigma_sQ^{1/2}) (\sigma_sQ^{1/2})^*$ for $s\in [0,T]$ takes values in $L_1(H,H)$.
\end{remark}
\begin{remark} 
The integral
$
\int_0^t\sigma_sQ\sigma_s^* ds
$
is interpreted as a Bochner integral in the space of Hilbert-Schmidt operators
$L_{\text{HS}}(H)$. Indeed, $\sigma_sQ\sigma_s^*$ is a linear operator on $H$, and we have
\begin{align*}
\int_0^t\EE[\Vert\sigma_s Q\sigma_s^*\Vert_{\text{HS}}]ds&=\int_0^t\EE[\Vert\sigma_sQ^{1/2}(\sigma_sQ^{1/2})^*\Vert_{\text{HS}}]ds
\\
&\leq\int_0^t\EE[\Vert\sigma_sQ^{1/2}\Vert_{\text{HS}}^2]ds<\infty,
\end{align*}
by appealing to Lemma \ref{lem:HS-banachalg} and the 
assumption on $\sigma$ being an integrable process with respect
to $W$.
This means that the Bochner integral is $a.s.$ defined. If we relax integrability to go beyond $L^2$, this argument fails, but we still have a well-defined Bochner integral as we can argue pathwise.
\end{remark}

\begin{remark}
\label{rem:martingale}
From \citet[Theorem 8.2, p.~109]{PZ2007} we deduce that 
the process $(M_t)_{t\geq 0}$ with 
\begin{align*}
    M_t=  \left(Y_t\right)^{\otimes 2}-\langle\langle Y\rangle\rangle_t
\end{align*}
is an $L_1(H)$-valued martingale w.r.t.~$(\mathcal{F}_t)_{t\geq 0}$. Thus, the operator angle bracket process can be called the {\it quadratic covariation process} of $Y_t$, which we shall do from now on.  
\end{remark}

%\subsection{Preliminary results for the Wiener process }

We end this section with a general expression for the even moments of an increment of the Wiener process. Later we will need the fourth moment in our analysis. 

%Define a generic increment as
%$\Delta W_t:=W_{t+\Delta}-W_t$ for $\Delta>0$ of the $Q$-Wiener process $W$ in $U$.
%We want to find  an analytic expression for the fourth moment of the norm of 
% $\Delta W_t$. 
 
 First, we introduce the $p$-trace of an operator $K\in L(U)$: We denote by $\text{Tr}_p(K)$ the 
{\it $p$-trace} of $K$, $p\in\mathbb N$, 
defined as
$$
\text{Tr}_p(K)=\sum_{i=1}^{\infty}\langle K e_i,e_i\rangle_U^p,
$$
whenever this converges. Here, $(e_i)_{i\in\mathbb N}$ is an ONB in $U$. We denote by $\text{Tr}$ the classical trace, given by $\text{Tr}=\text{Tr}_1$.
Consider now the positive definite symmetric trace class operator $Q$. If we organize the eigenvalues $(\lambda_i)_{i=1}^{\infty}\subset\mathbb R_+$ of $Q$ in decreasing order, letting $(e_i)_{i\in\mathbb N}$ be the ONB of eigenvectors, we have
$$
\text{Tr}_p(Q)\leq \lambda_1^{p-1}\sum_{i=1}^{\infty}\lambda_i=\text{Tr}(Q),
$$
and hence the $p$-trace is bounded by the trace for any $p>1$, and therefore also finite. The proof of the following result is relegated to Appendix \ref{App:proofs}:
\begin{lemma}
\label{lemma:4thmoment}
Let $W$ be a $Q$-Wiener process on $U$ and $q\in\mathbb N$ and define a generic increment as
$\Delta W_t:=W_{t+\Delta}-W_t$ for $\Delta>0$. Furthermore, let $(e_k)_{k\in\N}$ be the ONB in $U$ of eigenvectors of $Q$ with associated eigenvalues $(\lambda_k)_{k\in\N}$. Then, for any $t\geq 0$ and $m\in\N$ it holds that
$$
\mathbb E\left[\Vert\Delta W_t\Vert_U^{2q}\right]=(-i)^q\lim_{m\rightarrow\infty}\Phi_m^{(q)}(0),
$$
where
$$
\Phi_m(x)=\exp\left(-\frac12\sum_{k=1}^m\ln(1-2ix\Delta\lambda_k)\right),
$$
for $x\in\R$. 
In particular, 
$$
\mathbb{E}[\Vert \Delta W_t\Vert_U^4]=\Delta^2\left(\text{Tr}(Q)^2+2\text{Tr}_2(Q)\right).
$$
%where $\text{Tr}_p(Q)$ is the $p$-trace of $Q$, $p\geq 1$, 
%defined as
%$$
%\text{Tr}_p(Q)=\sum_{i=1}^{\infty}\langle Q e_i,e_i\rangle^p
%$$
%for an ONB $(e_i)_{i\in\mathbb N}$ in $H$. We denote by $\text{Tr}$ the classical trace, given by $\text{Tr}=\text{Tr}_1$.
\end{lemma}
This finishes our section with preliminary results.

\section{The weak law of large numbers}\label{sec: Weak Law of large numbers}

In this section, we show our main result on the law of large numbers for Volterra-type stochastic integrals in Hilbert space with 
operator-valued volatility processes. 

Consider 
\begin{equation}\label{Volatility Integral}
    Y_t:=\int_0^t\mathcal S(t-s)\sigma_sdW_s,
\end{equation}
where $W$ is a $Q$-Wiener process on the separable Hilbert space $U$, $\sigma$ is an element of $\mathcal L_{2,T}(U,H)$ and $\mathcal S$ is a $C_0$-semigroup on $H$. 
We assume that we observe $Y$ at times $t_i:=i\Delta_n$ for $\Delta_n=1/n$, $i=1,\dots,\lfloor t/\Delta_n\rfloor$ and define the
semigroup-adjusted increment
\begin{equation}\label{Adjusted Increment}
    \widetilde{\Delta}_n^iY:=Y_{t_i}-\mathcal S(\Delta_n)Y_{t_{i-1}}=\int_{t_{i-1}}^{t_i}\mathcal S(t_i-s)\sigma_sdW_s.
\end{equation}
We define the process of the semigroup-adjusted realised covariation (\SRV) as 
\begin{align*}
 t\mapsto\sum_{i=1}^{\ul}(\tilde{\Delta}_n^iY)^{\otimes 2}.
\end{align*}
The aim is to prove the following weak law of large numbers for the \SRV
\begin{align*}
 \sum_{i=1}^{\ul}(\tilde{\Delta}_n^iY)^{\otimes 2} 
    \stackrel{ucp}{\rightarrow}  
 \int_0^t\sigma_sQ\sigma_s^*ds, \qquad \text{ as } n \to \infty,
\end{align*}
in the ucp-topology, 
that is, for all $\epsilon>0$ and $T>0$
\begin{align}\label{ucp convergence with semigroup}
\lim_{n\to\infty}\mathbb{P}\left(\sup_{0\leq t\leq T}\left\| \sum_{i=1}^{\ul}(\tilde{\Delta}_n^iY)^{\otimes 2} - \int_0^t\sigma_sQ\sigma_s^*ds\right\|_{\text{HS}}\right)=0.
\end{align}

%\begin{remark}
 %Based on these results, we refer to $t\to\sum_{i=1}^{\ul}(\tilde{\Delta}_n^iY)^{\otimes 2}$ as the realized covariation process of $Y$. 
%\end{remark}

\subsection{Technical assumptions}\label{sec: Technical assumptions}
We need some technical assumptions on the stochastic volatility process $\sigma$. 
\begin{assumption}\label{as:smoothvol}
Assume that the volatility process satisfies the following H\"older continuity property: For all $T>0$ and $s, t \in [0,T]$ we have
$$\mathbb E\left[\Vert(\sigma_t-\sigma_s)Q^{\frac 12}\Vert^2_{\text{HS}}\right]^{\frac 12} \leq C_1(T)|t-s|^{\alpha},
$$
for some $\alpha>0$
and a constant $C_1(T)>0$ (depending on $T$).
\end{assumption}
Notice that we assume only local mean-square-H{\"o}lder continuity for the paths of the volatility process. This allows for including volatility processes with  c{\`a}dl{\`a}g paths in our considerations, as we will see later. 

We shall also need a moment condition to hold for the volatility process:
\begin{assumption}
\label{as:fourthmomentvol}
Assume that the volatility process satisfies for all $T>0$ the following moment conditions:
\begin{equation}
    \EE\left[\Vert\sigma_sQ^{\frac 12}\Vert^
     4_{\text{HS}}\right]\leq C_2(T)\quad \forall s \in [0,T], 
\end{equation}
for some constant $C_2(T)>0$ (depending on $T$).
\end{assumption}
 
 \begin{remark}
 %Observe that the Assumptions \ref{as:smoothvol} and \ref{as:fourthmomentvol} together imply {\bf Fred: I do not understand what is claimed here, as the following is trivial from Assumption 2????}
% \begin{equation}
 %  \sup_{s \in [0,T]} \EE\left[\Vert\sigma_s\Vert^
%     4_{\text{op}}\right]\leq C_2(T),
%\end{equation}
%since continuous functions are bounded on compacts.
 Using the Cauchy-Schwarz inequality, we can deduce under Assumption \ref{as:fourthmomentvol} for each $T>0$ 
 \begin{align*}
  \sup_{s \in [0,T]} \EE\left[\Vert\sigma_sQ^{\frac 12}\Vert^
     2_{\text{HS}}\right]
     & \leq  \sup_{s \in [0,T]} \sqrt{\EE\left[\Vert\sigma_sQ^{\frac 12}\Vert_{\text{HS}}^4\right]}
     \leq \sqrt{C_2(T)}.
 \end{align*}
 Moreover, we find  that for all $t\in [0,T]$, also
\begin{align*}
    \EE\left[\int_0^t\Vert\sigma_sQ^{1/2}\Vert^2_{\text{HS}}ds\right] \leq  t \sqrt{C_2(T)}<\infty.
\end{align*}
Thus, the integrability condition on $(\sigma_t)_{t\in[0,T]}$ holds for adapted processes satisfying Assumption \ref{as:fourthmomentvol}.
 \end{remark}

The semigroup is in general not continuous with respect to time in the operator norm, but only strongly continuous. This makes it more involved to verify convergence in Hilbert-Schmidt norms, like \eqref{ucp convergence with semigroup}, since then the semigroup component $S(\Delta_n)$ in the adjusted increment \eqref{Adjusted Increment} converges just strongly to the identity. 
However, we can make use of compactness of the closure of the image of the operators $\sigma_sQ^{\frac 12}$ for each $s\in [0,T]$,
and show the convergence of the semigroup to the identity operator on compacts by the subsequent argument in Theorem \ref{C: Application of Arzela Ascoli}.
This line of argument necessitates one of the following two alternative assumptions:
\begin{assumption}
\label{as:Q is more than Hilbert Schmidt}
\begin{itemize}
\item[(a)]Assume we can find a mean-square continuous process $ (\mathcal{K}_s)_{s\in \mathbb{R}_+}\in L^2(\Omega\times\mathbb{R}_+;L(U,H))$ of compact operators and a Hilbert-Schmidt operator $\mathcal{T}\in L_{{\text HS}}(U)$ such that almost surely $\sigma_sQ^{\frac 12}=\mathcal{K}_s\mathcal{T}$ for each $s\in [0,t]$.  
\item[(b)] The semigroup $(S(t))_{t\geq 0}$ is uniformly continuous, that is $S(t)=e^{A t}$ for some bounded operator $A\in L(H)$.
\end{itemize}
\end{assumption}
Observe, that Assumption \ref{as:Q is more than Hilbert Schmidt}(a) is fulfilled in the following cases:
\begin{itemize}
\item[(i)] $\sigma$ satisfies Assumption \ref{as:smoothvol} and $\sigma_t$ is almost surely compact (for instance itself a Hilbert-Schmidt operator) for each $t\in[0,T]$. In this case we can choose $\mathcal{K}_s:=\sigma_s$ and $\mathcal{T}:=Q^{\frac 12}$.
\item[(ii)] $\sigma$ satisfies Assumption \ref{as:smoothvol} and there exists an $\epsilon>0$, such that $Q^{(1-\epsilon)}$ is still a nuclear operator, that is, the eigenvalues of $Q$ satisfy $\sum_{n\in\mathbb{N}}\lambda_n^{1-\epsilon}<\infty$. In this case we can choose $\mathcal{K}_s:=\sigma_sQ^{\frac{\epsilon}{2}}$ and $\mathcal{T} :=Q^{\frac{1-\epsilon}{2}}$. Notice that this eigenvalue-property on $Q$ is not always fulfilled. We could for example have an operator with eigenvalues $\lambda_n^{\frac 12}:=n^{\frac 12 +\frac 1n}$
\end{itemize}

\begin{remark}
The semigroup given by $S(t)=I$ for all $t\geq 0$, where $I$ is the identity operator, is uniformly continuous and therefore satisfies Assumption 3(b). %Thus, the following proves hold under Assumptions \ref{as:smoothvol} and \ref{as:fourthmomentvol} and  when $Y$ is given as the martingale
%\begin{equation}
%    Y(t):=\int_0^t\mathcal \sigma_sdW_s.
%\end{equation}
\end{remark}

\subsection{The main result}\label{sec: Main result}

In order to prove the ucp-convergence (\ref{ucp convergence with semigroup})  we will first show the following stronger result:

\begin{theorem}\label{T: LLN for semigroup case}
Assume that Assumptions \ref{as:smoothvol}, \ref{as:fourthmomentvol} and either \ref{as:Q is more than Hilbert Schmidt}(a) or \ref{as:Q is more than Hilbert Schmidt}(b) hold.
For each $T>0$ there is a constant $L(T)>0$ such that  
\begin{align}\label{L:Convergence speed}
 \mathbb{E}\left[\sup_{0\leq t\leq T}\left\| \sum_{i=1}^{\ul}(\tilde{\Delta}_n^iY)^{\otimes 2} - \int_0^t\sigma_sQ\sigma_s^*ds\right\|_{\text{HS}}\right]
\leq  L(T) (\Delta_n^{\alpha}+ b_n^{\frac 12}(T)),
\end{align}
where
\begin{align}\label{Convergence Rate sequence}
b_n(T):= \sup_{r\in [0,T]}\mathbb E[ \sup_{x\in [0,\Delta_n]}\Vert
(I-\mathcal S(x))\sigma_rQ^{\frac 12}\|_{op}^2].
\end{align}
In particular, for all $T>0$
$$\lim_{n\to\infty}\mathbb{E}\left[\sup_{0\leq t\leq T}\left\| \sum_{i=1}^{\ul}(\tilde{\Delta}_n^iY)^{\otimes 2} - \int_0^t\sigma_sQ\sigma_s^*ds\right\|_{\text{HS}}\right]=0.$$
\end{theorem}
Before we prove this result in Section \ref{sec: Proofs}, we will make a couple of remarks and discuss uniform continuity of semigroups on compact sets. 
\begin{remark}
The factor $L(T)$ in the Theorem above is actually not just depending on $T$, but also shrinks when $n$ gets larger. Effectively, the constant can be precisely computed by careful inspection of the estimates \eqref{convergence inequality for first summand}, \eqref{convergence inequality for second summand}, \eqref{convergence inequality for third summand} and \eqref{convergence inequality for fourth summand} in the proof of Theorem \ref{T: LLN for semigroup case}. However, the expression becomes rather extensive and we refrain from stating it here. 

 That $(b_n)_{n\in\mathbb N}$ converges to $0$ is an implication of the following Proposition \ref{C: Application of Arzela Ascoli}. The magnitude of this sequence essentially determines the rate of convergence of the realised covariation by virtue of
inequality \eqref{L:Convergence speed}. We will come back to the magnitude of the  $b_n$'s in specific cases in Section \ref{sec: Convergence rates and semigroups}. 
\end{remark}

Denote for $t\geq 0$
\begin{equation}\label{Global Bound for the semigroup}
M(t):=\sup_{x\in[0,t]}\|S(x)\|_{op},
\end{equation}
 which is finite by the Hille-Yosida bound on the semigroup.
Often in stochastic modelling one also has a drift present. The following remark shows that our results are not altered by this:
\begin{remark}\label{R:Drift extension}
Observe that we could easily extend $Y$ to posses a drift and an "inital condition", that is
$$Y_t=\mathcal{S}(t)h+\int_0^t\mathcal{S}(t-s)\alpha_s ds +\int_0^t \mathcal S (t-s) \sigma_s dW_s,$$
for a predictable and almost surely Bochner-integrable stochastic process $(\alpha_t)_{t\in [0,T]}$, such that 
\begin{equation}\label{Finite moment condition for the drift}
    \sup_{t\in[0,T]}\mathbb E [\|\alpha_t\|^2]<\infty,
\end{equation}
and for an initial value $h\in H$.
In this case
\begin{align*}\label{Adjusted Increment}
    \widetilde{\Delta}_n^iY:= Y_{t_i}-\mathcal S(\Delta_n)Y_{t_{i-1}}
    =  \int_{t_{i-1}}^{t_i}\mathcal S(t_i-s)\alpha_s ds+\int_{t_{i-1}}^{t_i}\mathcal S(t_i-s)\sigma_sdW_s.
\end{align*}
We can then argue that 
\begin{align*}
 & \mathbb{E}\left[\sup_{0\leq t\leq T}\left\| \sum_{i=1}^{\ul}(\tilde{\Delta}_n^iY)^{\otimes 2} - \int_0^t\sigma_sQ\sigma_s^*ds\right\|_{\text{HS}}\right]\\
&\qquad\leq  \mathbb{E}\left[\sup_{0\leq t\leq T}\left\| \sum_{i=1}^{\ul}\left(\int_{t_{i-1}}^{t_i}\mathcal S(t_i-s)\alpha_s ds\right)^{\otimes 2}\right\|_{\text{HS}}\right]\\
&\qquad\qquad + \mathbb{E}\left[\sup_{0\leq t\leq T}\left\| \sum_{i=1}^{\ul}\left(\int_{t_{i-1}}^{t_i}\mathcal S(t_i-s)\sigma_sdW_s\right)^{\otimes 2} - \int_0^t\sigma_sQ\sigma_s^*ds\right\|_{\text{HS}}\right]\\
&\qquad=(1)+(2)
\end{align*}
Summand $(2)$ can be estimated with Theorem \ref{T: LLN for semigroup case}. For Summand $(1)$ %=\mathcal{O}(\Delta_n)$ 
we find
\begin{align*}
   & \mathbb{E}\left[\sup_{0\leq t\leq T}\left\| \sum_{i=1}^{\ul}\left(\int_{t_{i-1}}^{t_i}\mathcal S(t_i-s)\alpha_s ds\right)^{\otimes 2}\right\|_{\text{HS}}\right]\\
   &\qquad\leq \mathbb{E}\left[ \sum_{i=1}^{\lfloor T/\Delta_n\rfloor}\left\| \int_{t_{i-1}}^{t_i}\mathcal S(t_i-s)\alpha_s ds\right\|_{H}^2\right]\\
  &\qquad\leq  \sum_{i=1}^{\lfloor T/\Delta_n\rfloor} \Delta_n^2 M^2(T) \sup_{r\in[0,T]}\mathbb{E}\left[\|\alpha_r\|_H^2\right]\\
  &\qquad\leq M^2(T) T \sup_{r\in[0,T]}\mathbb{E}\left[  \|\alpha_r\|_H^2\right]\Delta_n,
\end{align*}
where we appealed to the bound \eqref{Finite moment condition for the drift} on the semigroup. 
Hence, Summand $(1)$ is $\mathcal{O}(\Delta_n)$ and will not impact the estimation of the covariation (in the limit).
\end{remark}

\subsection{Extension by localisation}\label{sec: Extension by localization}

In general, we have the following result:
\begin{theorem}
\label{T: Extension by localization}
Let $(\Omega_m)_{m\in \mathbb N}$ be a sequence of measurable subsets such that $\Omega_m\uparrow \Omega$. Suppose Assumptions \ref{as:smoothvol}, \ref{as:fourthmomentvol} and \ref{as:Q is more than Hilbert Schmidt} hold for $\sigma^{(m)}:=\sigma \mathbf{1}_{\Omega_m}$ for all $m\in \mathbb N$. Then
\begin{equation}
\label{eq:final-result}
    \lim_{n\rightarrow\infty}\Prob\left(\sup_{0\leq s\leq t}\left\Vert\sum_{i=1}^{\uls}(\tilde{\Delta}_i^n Y)^{\otimes 2}-\int_0^s\sigma_u Q\sigma_u^*du\right\Vert_{\text{HS}}>\epsilon\right)=0,
\end{equation}
for any $\epsilon>0$, that is, convergence holds in $ucp$ of the realized covariation. 
\end{theorem}

We can apply the localization on volatility processes $\sigma$ with almost sure H\"older-continuous paths:
\begin{corollary}\label{C: Localization for almost surely Holder continuous functions}
Assume $\sigma$ is almost surely $\alpha$-H{\"o}lder-continuous on $[0,T]$ with respect to the operator norm, satisfies Assumption \ref{as:Q is more than Hilbert Schmidt} and that the initial value has a finite fourth moment, i.e.
\begin{align}\label{Fourth moment of initial state is finite}
\mathbb E[\|\sigma_0\|_{\text{op}}^4]<\infty.
\end{align} Then the ucp convergence in Eq. \eqref{eq:final-result} holds.
\end{corollary}
\begin{proof}
%For $\alpha >0$ let $Y:=(Y_t)_{t\in [0,T]}$ be an almost surely $\alpha$-H{\"o}lder continuous process with values in a separabel Hilbert space $L_{\text{HS}}(H)$ of Hilbert Schmidt opertors. 
We know that  
\begin{align}
 C(T):=  \sup_{s\neq t\in [0,T]} \frac{\| (\sigma_t-\sigma_s)Q^{\frac 12}\|_{\text{HS}}}{|t-s|^{\alpha}}<\infty, \qquad \mathrm{a.~s.}
\end{align}
Then $C(T)$ is a random variable and the set $\Omega_m:= \lbrace \omega \in \Omega: C(T)\leq m\rbrace$ is measurable and $\Omega_m\uparrow \Omega$\footnote{At least a convergence to a set with full measure}. We have to verify that $\sigma^{(m)}=\sigma \mathbf{1}_{\Omega_m}$ fulfills Assumptions \ref{as:smoothvol} and \ref{as:fourthmomentvol}, since \ref{as:Q is more than Hilbert Schmidt} is satisfied automatically. The $\alpha$-H{\"o}lder continuity is obtained since
\begin{align*}
\mathbb{E}[\|(\sigma_t^{(m)}-\sigma_s^{(m)})Q^{\frac 12}\|_{\text{HS}}^2]\leq m^2 |t-s|^{2\alpha} \text{Tr}(Q),
\end{align*}
and the fourth moment is finite since 
\begin{align*}
\mathbb{E}[\|\sigma_t^{(m)} \|_{\text{op}}^4]\leq
\mathbb{E}[\|\sigma_t^{(m)} -\sigma_0^{(m)}\|_{\text{op}}^4]+\mathbb{E}[\|\sigma_0^{(m)} \|_{\text{op}}^4]
\leq m^4 t^{4\alpha}+\mathbb{E}[\|\sigma_0 \|_{\text{op}}^4]<\infty.
\end{align*}
The proof is complete.
\end{proof}

\section{Applications}\label{sec: Applications}

In this section, we give an overview of potential settings and scenarios for which we can use the techniques described above to infer volatility.

Stochastic integrals of the form (\ref{Volatility Integral}) arise naturally in correspondence to mild or strong solutions to stochastic partial differential equations. Take as a simple example a process given by 
\begin{equation}\label{SPDE}
(\text{SPDE})\begin{cases}
dY_t=AY_t dt + \sigma_tdW_t , \qquad t \geq 0\\
Y_0=h_0\in H,
\end{cases}
\end{equation}
where $A$ is the generator of a $C_0$-semigroup $(\mathcal S(t))_{t\geq 0}$ on the separable Hilbert space $H$, $W$ is a $Q$-Wiener process on a separable Hilbert space $U$ for some positive semidefinite and symmetric trace class operator $Q:U\to U$ and $\sigma \in \mathcal{L}_{T,2}(U,H)$.

There are three components in this model, which need to be estimated in practice:  the covariance operator $Q$ of the Wiener process, the generator $A$ (or the semigroup $(\mathcal S(t))_{t\geq 0}$ respectively) and the stochastic volatility process $\sigma$.

\subsection{Semigroups}\label{sec: Convergence rates and semigroups}
The essence of the convergence result
in Theorem \ref{T: LLN for semigroup case} is that we can infer on $Q$ and $\sigma$ based on observing the path of $Y$, given that we {\it know} the semigroup $(\mathcal S(t))_{t\geq 0}$. 
Certainly, this is not always the case, since we may just have knowledge about the infinitesimal generator $A$. However, if we know the precise form of the semigroup it is sometimes possible to estimate the speed of convergence, that is, a bound on the $b_n(T)$'s given in \eqref{Convergence Rate sequence}.

\subsubsection{Martingale case}\label{sec: martingale case}

For $A=0$ and $S(t)=I$ and for all $t\geq 0$, we have the solution
$$Y_t=\int_0^t \sigma_s dW_s,$$
for the stochastic partial differential equation (\ref{SPDE}).
Clearly in this case we have
$$b_n(T)=0.$$

\subsubsection{Uniformly continuous semigroups}
Assume that $(\mathcal  S(t))_{t\geq0}$ is continuous with respect to the operator norm. This is equivalent to $A\in L(H)$ and $\mathcal  S(t)=e^{t A}$. 

\begin{lemma}
If the semigroup $(\mathcal S (t))_{t\geq 0}$ is uniformly continuous, we have, for $b_n$ given in \eqref{Convergence Rate sequence}, that 
$$b_n(T)\leq \Delta_n \|A\|_{\text{op}} e^{\|A\|_{\text{op}}\Delta_n}\sup_{r\in[0,T]}\mathbb E [\|\sigma_rQ^{\frac 12}\|^2_{\text{HS}}].$$
In particular, if Assumptions \ref{as:smoothvol} and \ref{as:fourthmomentvol} are valid, we have
$$b_n(T)\leq \Delta_n \|A\|_{\text{op}} e^{\|A\|_{\text{op}}\Delta_n}\sqrt{C_2(T)}\text{Tr}(Q).$$
\end{lemma}

\begin{proof}
Recall the following fundamental equality from semigroup theory (cf. \citet[Lemma 1.3]{Engel1999}): 
\begin{align}\label{Fundamental Theorem of Semigroup Theory}
(\mathcal S(x)-I)h= & \int_0^x A \mathcal S (s) h ds,\quad\quad \forall h\in H\\
= & \int_0^x  \mathcal S (s)A h ds,\quad\quad \forall h\in D(A).\label{Fundamental Theorem of Semigroup Theory II}
\end{align}
Using \eqref{Fundamental Theorem of Semigroup Theory}, we get 
\begin{align*}
 \sup_{x\in [0,\Delta_n]}\left\| (I-\mathcal S (x))\right\|_{\text{op}}= & \sup_{x\in [0,\Delta_n]}\sup_{\|h\|=1}\left\| \int_0^x A \mathcal S (s)  hds\right\|_{H}\\
 \leq & \sup_{x\in [0,\Delta_n]} x \|A\|_{\text{op}} e^{\|A\|_{\text{op}}x}=\Delta_n \|A\|_{\text{op}} e^{\|A\|_{\text{op}}\Delta_n}.
\end{align*}
%Since the operator norm is bounded by the Hilbert-Schmidt norm, we have 
%$$
%\|Q^{\frac 12}\|^2_{\text{op}}\leq\|Q^{\frac 12}\|^2_{\text{HS}}=\text{Tr}(Q).
%$$ 
It follows that
\begin{align*}
    b_n^2(T)=&\sup_{r\in [0,T]}\mathbb E [\sup_{x\in [0,\Delta_n]}\| (I-\mathcal S (x))\sigma_rQ^{\frac 12}\|^2_{\text{op}}] \\
    \leq & \sup_{x\in [0,\Delta_n]}\| (I-\mathcal S (x))\|_{\text{op}}^2\sup_{r\in [0,T]}\mathbb E [\|\sigma_rQ^{\frac 12}\|^2_{\text{HS}}]\\
    \leq &\Delta_n^2 \|A\|_{\text{op}}^2 e^{2\|A\|_{\text{op}}\Delta_n}\sup_{r\in [0,T]}\mathbb E [\|\sigma_rQ^{\frac 12}\|^2_{\text{HS}}].
\end{align*}
\end{proof}
For uniformly continuous semigroups we obtain a convergence speed of the order $\min(\Delta_n^{\frac 12},\Delta_n^{\alpha})$ for the convergence of the realized covariation to the quadratic covariation in Theorem \ref{T: LLN for semigroup case}.
\begin{remark}
Note that, if the semigroup is uniformly continuous and under Assumptions \ref{as:smoothvol} and \ref{as:fourthmomentvol}, we can get back to the martingale case of Section \ref{sec: martingale case} if we operate on the values of $Y_t$ in any of the following two ways:
\begin{itemize}
    \item[(i)] We continue as in the martingale case for the realised covariation of $\Tilde{Y_t}:=\mathcal S(-t)Y_t$: This can be done since $\mathcal S(t)=\exp(A t)$ and we have
$$
Y_t=\int_0^t \mathcal S (t-s) \sigma_s ds=\mathcal S(t)\int_0^t\mathcal S(-s)\sigma_sdW_s.
$$
Thus $\widetilde{Y}_t$ is a martingale.

\item[(ii)] We continue as in the martingale case for the realised covariation of $\Tilde{Y_t}:= Y_t-AY_t$: This can be done
since the process $(Y_t)_{t\in [0,T]}$
is the strong solution to \eqref{SPDE} with the continuous linear generator $A$ and $h_0\equiv 0\in H$, since $D(A)=H$ (see for instance Theorem 3.2 in \cite{GM2011}). That is, in particular,
$$Y_t=AY_t+\int_0^t \sigma_s dW(s),\quad\forall t\in [0,T].$$
\end{itemize}
\end{remark}

Let us turn our attention to a case of practical interest coming from financial mathematics applied to commodity markets. 

\subsubsection{Forward prices in commodity markets: the Heath-Jarrow-Morton approach}
\label{subsect:hjmm}
A case of relevance for our analysis is inference on the volatility for forward prices in commodity markets as well as for forward rates in fixed-income markets.
The Heath-Jarrow-Morton-Musiela equation (HJMM-equation) describes the term structure dynamics in both of these settings  (see \cite{Filipovic2000} for a detailed motivation for the use in interest rate modelling and \cite{BenthKruhner2014} for its use in commodity markets) and is given by 
\begin{equation}\label{HJMM}
(\text{HJMM})\begin{cases}
dX_t=(\frac{d}{dx}X_t+\alpha_t )dt + \sigma_t dW_t , \qquad t \geq 0\\
X_0=h_0\in H,
\end{cases}
\end{equation}
where $H$ is a Hilbert space of functions $f:\mathbb{R}_+\to\mathbb R$ (the \textit{forward curve space}), $(\alpha_t)_{t\geq 0}$ is a predictable and almost surely locally Bochner-integrable stochastic process and $\sigma$ and $W$ are as before.
Conveniently, the states of this {\it forward curve dynamics} are realized 
%in terms of its forward rate dynamics modelled, for $\beta > 0$ fixed, 
on the separable Hilbert space 
\begin{align}\label{FCS}
H=H_{\beta} &= \left\{ h:\mathbb{R}_+\to\mathbb{R}: h\text{ is absolutely continuous 
and } \| h \|_{\beta} < \infty \right\},
\end{align}
for fixed $\beta>0$, where the inner product is given by
\begin{align*}
\langle h,g \rangle_{\beta} &= h(0)g(0) + \int_0^{\infty} h'(x)g'(x)\mathrm{e}^{\beta x}dx,
\end{align*}
and norm $\| h \|_{\beta}^2 = \langle h, h \rangle_{\beta}$. This space was introduced and analysed in \cite{Filipovic2000}. As in \cite{Filipovic2000}, one may consider more general scaling functions in the inner product than the exponential $\exp(\beta x)$. However, for our purposes here this choice suffices.
%{\bf Dennis, do we really need the complexification now, as we do not deal with the Riesz basis anymore?}As in \cite{BenthKruhner2014} we can embed $H_{\beta}$ into the larger Hilbert space
%\begin{align}\label{CFCS}
%\tilde{H}_{\beta} &= \left\{ h:\mathbb{R}_+\to\mathbb{C}: h\text{ is absolutely continuous 
%and } \| h \|_{\beta} < \infty \right\},
 %\\ \notag \langle h,g \rangle_{\beta} &= h(0)\bar{g}(0) + \int_0^{\infty} h'(x)\bar{g}'(x)\mathrm{e}^{\beta x}dx.
%\end{align}
%Here, $\bar{g}$ means complex conjugation.
The suitability of this space is partially due to the following result:
\begin{lemma}
The differential operator $A=\frac{d}{dx}$ is the generator of the strongly continuous semigroup $(\mathcal{S}(t))_{t\geq 0}$ of shifts on $H_{\beta}$, given by $\mathcal S(t)h(x)=h(x+t)$, for $h\in H_{\beta}$.
\end{lemma}
\begin{proof}
See for example \cite{Filipovic2000}.
\end{proof}

The HJMM-equation \eqref{HJMM} possesses a mild solution (see e.g. \cite{PZ2007})
\begin{align}\label{HJM-mild solution}
f_t=\mathcal S(t)f_0+\int_0^t \mathcal{S}(t-s) \alpha_s ds+\int_0^t \mathcal{S}(t-s) \sigma_s dW_s.
\end{align}

Since forward prices and rates are often modelled under a risk neutral probability measure, the drift has in both cases (commodities and interest rates) a special form. In the case of forward prices in commodity markets, it is zero under the risk neutral probability, whereas in interest rate theory it is completely determined by the volatility via the no-arbitrage drift condition
\begin{equation}\label{HJM Drift}\alpha_t= \sum_{j\in \mathbb{N}}\sigma_t^j \Sigma_t^j, \quad \forall t\in [0,T],
\end{equation}
where  $\sigma_t^j=\sqrt{\lambda_j}\sigma_t(e_j)$ and $\Sigma^j_t=\int_0^t\sigma^j_sds$ for some eigenvalues $(\lambda_j)_{j\in\mathbb N}$ and a corresponding basis of eigenvectors $(e_j)_{j\in\mathbb N}$ of the covariance operator $Q$ of $W$ (cf. Lemma 4.3.3 in \cite{Filipovic2000}).
\begin{lemma}
Assume that the volatility process $(\sigma_t)_{t\in[0,T]}$  satisfies Assumption \ref{as:fourthmomentvol} and that for each $t\in[0,1]$ the operator $\sigma_t$
maps into $$H_{\beta}^0=\lbrace h\in H_{\beta}: \lim_{x\to\infty} h(x)=0\rbrace.$$ Then the drift given by \eqref{HJM Drift} has values in $H_{\beta}$, is predictable, satisfies \eqref{Finite moment condition for the drift} and is almost surely Bochner integrable. Thus, the conditions of Remark \ref{R:Drift extension} are satisfied.
\end{lemma}
\begin{proof}
 That the drift is well defined follows from Lemma 5.2.1 in \cite{Filipovic2000}. Predictability follows immediately from the predictability of the volatility.
We have
 by Theorem 5.1.1 from \cite{Filipovic2000} that there is a constant $K$ depending only on $\beta$ such that
$$\|\sigma^j_t \Sigma^j_t\|_{\beta}\leq K \|\sigma_t^j\|_{\beta}^2.$$
Therefore, we get by the triangle inequality that
\begin{align*}
    \|\alpha_t\|_{\beta}\leq & K \sum_{j\in\mathbb N}\|\sigma_t^j\|_{\beta}^2=K\|\sigma_t Q^{\frac 12}\|_{\text{HS}}^2.
\end{align*}
Using Cauchy-Schwarz inequality we obtain
\begin{align*}
    \sup_{t\in[0,T]}\mathbb E[\|\alpha_t\|_{\beta}^2]\leq \sup_{t\in[0,T]}\mathbb E [\|\sigma_t Q^{\frac 12}\|_{\text{HS}}^4],
\end{align*}
which is finite by Assumption \ref{as:fourthmomentvol}. This shows \eqref{Finite moment condition for the drift}. Moreover, the Bochner integrability follows, since we have the stronger
$$\mathbb E \left[\int_0^T \|\alpha_t\|_{\beta}dt\right]
%=\int_0^T \mathbb E [\|\alpha_t\|_{\beta}]dt
\leq\int_0^T \mathbb E [\|\alpha_t\|_{\beta}^2]^{\frac 12}dt \leq T \sup_{t\in[0,T]}\mathbb E [\|\sigma_tQ^{\frac 12}\|_{\text{HS}}^4]^{\frac 12}<\infty.$$
The result follows. 
\end{proof}

\begin{remark}
Since we know the exact form of the semigroup $(S(t))_{t\geq 0}$, we can recover the adjusted increments $\tilde{\Delta}_n^if$ efficiently from forward curve data by a simple shifting in the spatial (e.g., time-to-maturity) variable of these curves. 
Theorem \ref{T: LLN for semigroup case} (and Remark \ref{R:Drift extension} in case of a nonzero drift in interest rate theory)  can therefore be applied in practice to make inference on $\sigma$ under Assumptions \ref{as:smoothvol}, \ref{as:fourthmomentvol} and \ref{as:Q is more than Hilbert Schmidt}, in which case the ucp-convergence \eqref{ucp convergence with semigroup} holds.
\end{remark}

The shift semigroup is strongly, but not uniformly, continuous, leaving us with the question to determine the convergence speed of the estimator established in Corollary \ref{L:Convergence speed}. 
We close this subsection by deriving a convergence bound under regularity condition of the volatility in the space variable (that is time to maturity).

Observe that by Theorem 4.11 in \cite{BenthKruhner2014} %and the fact that for interest rates $ran(\sigma)\subset H_{\beta}^0$ 
we know that for all $r\in[0,T]$ there exist random variables $c_r$ with values in $ \mathbb R$, $f_r,g_r$ with values in $ H$ such that $g_r(0)=0=f_r(0)$ and $p_r$ with values in $ L^2(\mathbb{R}^2_+)$ such that we have 
$$\sigma_r Q^{\frac 12} h(x)= c_r h(0)+\langle g_r, h\rangle_{\beta}+h(0)f_r(x)+ \int_0^{\infty} q_r(x,z)h'(z)dz,$$
where $q_r(x,z)=\int_0^x p_r(y,z) e^{\frac{\beta}{2}z-y}dy$.
We denote by $C_{\text{loc}}^{1,\gamma}:=C_{\text{loc}}^{1,\gamma}(\mathbb{R}_+)$ the space of continuously differentiable functions with locally  $\gamma$-H{\"o}lder continuous derivative for $\gamma \in(0,1]$.

%If $ran(\sigma_r Q^{\frac 12})\in C^{1,\gamma}$ for some $\gamma\in (0,1)$ we necessarily have that $h\in C^{1,\gamma}_{\text{loc}}$ and $q_r(\cdot,z)\in C^{1,\gamma}_{\text{loc}}$ for all $z\in\mathbb{R}_+$, which means that $p_r(\cdot,z)$ is locally $\gamma$-H{\"o}lder continuous. 
\begin{theorem}\label{T: Convergence rate for forward curves}
%{\bf Dennis, is the space $C_{loc}^{\beta}$ anywhere defined in our paper?}
Assume that $f_r,q_r(\cdot,z)\in C^{1,\gamma}_{\text{loc}}$ for all $z\geq 0$, $r\in[0,T]$ and that the corresponding local H{\"o}lder constants $L_r^1(x)$ of $e^{\frac{\beta}{2}\cdot}f_r'(\cdot)$ and $L^2_r(x,z)$ of $p_r$ are square integrable in $x$ and in $(x,z)$ respectively such that $$\hat L: =\sup_{r\in [0,T]}\mathbb E
\left[\left(|f_r'(\zeta)|+\|L_r^1\|_{L^2(\mathbb{R}_+)}+ \|L_r^2\|_{L^2(\mathbb{R}_+^2)}+  \frac{\beta}{2}    \|  p_r\|_{L^2(\mathbb{R}_+^2)}\right)^2\right]<\infty.$$ 
Then for $b_n(T)$ as given in \eqref{Convergence Rate sequence}, 
we can estimate
$$b_n(T)\leq \hat L \Delta_n^{2\gamma}.$$
\end{theorem}

In the next section, we investigate the validity of assumptions for volatility models.

\subsection{Stochastic volatility models}\label{sec: Stochastic volatility models}

In this section different models for stochastic volatility in Hilbert spaces are discussed.
So far, infinite-dimensional stochastic volatility models are specified by stochastic partial differential equations on the positive cone of Hilbert-Schmidt operators (see \cite{BenthRudigerSuss2018}, \cite{BenthSimonsen2018}). As such, Assumption \eqref{as:Q is more than Hilbert Schmidt} is trivially fulfilled. We will  check therefore, which models satisfy Assumptions \eqref{as:smoothvol} and \eqref{as:fourthmomentvol}.

Throughout this section, we take $H=U$ for simplicity. The volatility is oftentimes given as the unique positive square-root of a process
$\Sigma_t$, e.g., 
\begin{equation}
\sigma_t:=\Sigma^{\frac 12}_t,
\end{equation} 
where $\Sigma$ takes values in the set of positive 
Hilbert-Schmidt operators on $H$.%, given by some stochastic partial differential equation.

Before we proceed with the particular models, we state the following result:
\begin{lemma}\label{L:Squared Volatility Lemma}
Assume for some constants $\alpha, C_1(T)$ and $C_2(T)$ that for all $s,t\in[0,T]$ we have
\begin{equation}\label{Holder condition for squared Volatility}
 \mathbb E \left[\| (\Sigma_t-\Sigma_s)\|_{\text{op}}^2\right]^{\frac 12}
\leq \frac{C_1(T)^2}{\text{Tr}(Q)^2} (t-s)^{2\alpha}
\end{equation}
and 
\begin{equation}
   \sup_{s\in [0,T]} \mathbb E [\| \Sigma_s\|^2_{\text{op}}]\leq C_2(T).
\end{equation} 
Then $\sigma$ satisfies Assumptions \ref{as:smoothvol} and \ref{as:fourthmomentvol} with corresponding constants $\alpha, C_1(T)$ and $C_2(T)$.
\end{lemma}

\begin{proof}
By the inequality in Lemma 2.5.1 of \cite{Bogachev2018}, the H{\"o}lder inequality and (\ref{Holder condition for squared Volatility})
\begin{align*}
\mathbb{E}[\| (\sigma_t-\sigma_s)Q^{\frac 12}\|_{\text{HS}}^2]\leq &  \mathbb{E}[\| (\Sigma^{\frac 12}_t-\Sigma^{\frac 12}_s)\|_{\text{op}}^2] \text{Tr}(Q)\\
\leq & \mathbb{E}[\| (\Sigma_t-\Sigma_s)\|_{\text{op}}] \text{Tr}(Q)\\
\leq & \mathbb{E}[\| (\Sigma_t-\Sigma_s)\|_{\text{op}}^2]^{\frac 12} \text{Tr}(Q)\\
%\leq & \mathbb{E}[\| (\Sigma(t)-\Sigma(s))\|_{\text{HS}}^2]^{\frac 12} \text{Tr}(Q)\\
\leq & C_1(T) (t-s)^{\alpha}.
\end{align*}
Moreover, Assumption \ref{as:fourthmomentvol} is satisfied, since 
\begin{align*}
   \sup_{s \in [0,T]}\EE[\| \sigma_s\|^
     4_{\text{op}}]=   \sup_{s \in [0,T]}\EE[\|\Sigma^{\frac 12}_s\|^4_{\text{op}}]= \sup_{s \in [0,T]}\EE[\|\Sigma_s\|^2_{\text{op}}]\leq C_2(T).
\end{align*}
The proof is complete.
\end{proof}
\subsubsection{Barndorff-Nielsen \& Shephard (BNS) model}

We assume $\Sigma$ is given by the Ornstein-Uhlenbeck dynamics
\begin{align*}
(BNS)\begin{cases}
d\Sigma_t=\mathbb{B} \Sigma_tdt+ d\mathcal{L}_t,\\
\Sigma_0= \Sigma\in L_{\text{HS}}(H),
\end{cases}
\end{align*}
where $\mathbb B$ is a positive bounded linear operator on the space of Hilbert-Schmidt operators $L_{\text{HS}}(H)$ and $\mathcal{L}$ is a square integrable L{\'e}vy subordinator on the same space. $\mathbb{B}$ is then the generator of the uniformly continuous semigroup given by $\mathbb{S}(t)=\exp(\mathbb{B}t)$ and the equation has a mild solution given by
\begin{align*}
\Sigma_t=\mathbb{S}(t)\Sigma_0+\int_0^t \mathbb{S}(t-s)d\mathcal{L}_s,
\end{align*}
which defines a process in $\mathcal{L}_{T,2}(H,H)$
(see \cite{BenthRudigerSuss2018}).
Stochastic volatility models with OU-dynamics were suggested in \cite{BenthRudigerSuss2018}, extending the BNS-model introduced in \cite{Barndorff-Nielsen2001}
to infinite dimensions.

\begin{lemma}\label{L: Mean Square Lipschitz continuity of OU-Processes} For all $s,t\in [0,T]$ such that $t-s\leq 1$ we have
\begin{align*}
  \mathbb E [\| (\Sigma_t-\Sigma_s)\|_{\text{HS}}^2]^{\frac 12}
\leq \tilde{L}(T) (t-s)^{\frac 12},
\end{align*}
where we denote $$\tilde{L}(T):=  \sqrt{3} (\mathbb C e^{\|\mathbb C \|_{\text{op}}T} \|\Sigma_0\|_{\text{HS}} +e^{\|\mathbb C\|_{\text{op}}T} \text{Tr}(Q_{\mathcal L})^{\frac 12} (1+\mathbb C e^{\|\mathbb C \|_{\text{op}}T})  )\text{Tr}(Q). $$
In particular, $\sigma$ satisfies Assumptions \ref{as:smoothvol} and \ref{as:fourthmomentvol} with corresponding constants $\alpha= \frac 14 $, $C_1(T)= \sqrt{\tilde{L}(T)}\text{Tr}(Q)$ and $C_2(T)= e^{\|\mathbb C\|_{\text{op}} T}(\|\Sigma_0\|_{\text{HS}}+  \text{Tr}(Q)^{\frac 12} T^{\frac 12})$.
\end{lemma}

It is also possible to derive ucp convergence for rough volatility models, which we present in the following section.

%While the BNS-model in infinite dimensions directly models volatility as an Ornstein-Uhlenbeck process in the space of Hilbert-Schmidt operators, there is also the the possibility to "square" processes, which act in the Hilbert space of the actual diffusion process $Y$.  In this approach the induced volatility process is automatically positive and symmetric
%and it is presented in the next section in form of the infinite dimensional Heston model from \cite{BenthSimonsen2018}.

\subsubsection{Rough volatility models}

In \cite{BenthHarang2020} pathwise constructions of Volterra processes are established and suggested for the use in stochastic volatility models.
In this setting, a process is mostly known to be H{\"o}lder continuous almost surely of some particular order.

 Therefore we fix an almost surely  H{\"o}lder continuous process  $(Y_t)_{t\in [0,T]}$ of order $\alpha$ with values in $H$. 
Without any further knowledge of the process, we do not know whether the corresponding H{\"o}lder constant, that is the random variable $C(T)$ such that
\begin{align}\label{local Holder constant}
    C(T):=\sup_{s,t\in [0,T]}\frac{\|Y_t-Y_s\|_H}{|t-s|^{\alpha}},
\end{align}
is square-integrable, and therefore we cannot verify Assumptions \ref{as:smoothvol} or \ref{as:fourthmomentvol} without additional assumptions. However, for various models we can use Corollary \ref{C: Localization for almost surely Holder continuous functions}.
If $H$ is a Banach algebra (like the forward curve space defined by \eqref{FCS}), we can define the volatility process by
\begin{equation}\label{Rough exponential volatility}
    \sigma_t h:=\exp(Y_t) h.
\end{equation}
This is a direct extension of the volatility models proposed in \cite{Gatheral2018}. 
\begin{lemma}
Assume that $H$ be a commutative Banach algebra and $\sigma$ is defined by \eqref{Rough exponential volatility}. Moreover assume that $$\mathbb E [\exp(4\|Y_0\|_H)]<\infty.$$ Then the ucp-convergence in \eqref{ucp convergence with semigroup} holds.
\end{lemma}

\begin{proof}
Since in commutative Banach algebras $\exp(f+g)=\exp(f)\exp(g)$ holds for all $f,g\in H$, we have 
\begin{align*}
    \|\exp(f)-\exp(g)\|_{\text{op}}\leq & \exp(\|f-g\|_H)\|\exp(g)-\exp(-f+2g)\|\\\leq & 2\exp(2\|f\|_H+2\|g\|_H) \|f-g\|_H.
\end{align*}
This implies the local $\alpha$-H{\"o}lder continuity of $\sigma$. Due to Corollary \ref{C: Localization for almost surely Holder continuous functions} the assertion holds.
\end{proof}

%\textbf{Questions: How do we know anything about the Hölder regularity of this entitiy? Is the induced Hölder constant square integrable? Does it help to not consider the operator exponential but the multiplication operator of the exponential over a Hilbert algebra (eg Filipovic space), such that we can say something about Local lipschitzianity of the exponential function? Square integrability of Hölder constant for Gaussian Processes with something like Garsia-Rodemich-Rumsey inequality?}

%{\bf Dennis, what you find for the constant is probably, after Cauchy-Schwarz,
%\begin{align*}
%   \mathbb E[\|\exp(Y_t)-\exp(Y_s)\|_{\text{op}}]&\leq \EE[\exp(2\Vert Y_t\Vert_H+2\Vert Y_s\Vert_H)]^{1/2}\EE[\Vert Y_t-Y_s\Vert_H^2]^{1/2} \\
%   &\leq \EE[\exp(4\Vert Y_t\Vert_H)]^{1/4}\EE[\exp(4\Vert Y_s\Vert_H)]^{1/4}\EE[\Vert Y_t-Y_s\Vert_H^2]^{1/2}
%\end{align*}
%For the estimation of the two exponents of $\Vert Y\Vert_H$, maybe one could appeal to Fernique's theorem as you have Gaussianity here? Dennis: I still have a problem with stating that $C(T)$ is square integrable. Do I overlook something?
%}

\section{Proofs}\label{sec: Proofs}
In this section, we will present the proofs of our previously stated results. 
\subsection{Proofs of results in Section \ref{sec: Weak Law of large numbers}}

\subsubsection{Uniform continuity of semigroups on compact sets}
In order to verify that $b_n(T)$ defined in \eqref{Convergence Rate sequence} converges to $0$ and to prove Theorem \ref{T: LLN for semigroup case}, % in Section \ref{sec: Proofs}, 
we need to establish some convergence properties of semigroups on compacts.

Let $X$ be a compact Hausdorff space. Recall that a subset $F\subset C(X;\mathbb{R})$ is equicontinuous, if
for each $x\in X$ and $\epsilon>0$ there is a neighbourhood $U_x$ of $x$ in $X$ such that for all $y\in U_x$ and for all $f\in F$ we have
$$| f(x)-f(y)|\leq \epsilon.$$
$F$ is called pointwise bounded, if for each $x\in X$ the set $\lbrace |f(x)|: f\in F\rbrace$ is bounded in $\mathbb{R}$. $F$ is called relatively compact (or conditionally compact), if its closure is compact.
For convenience, we recall the Arzel\'{a}-Ascoli Theorem (see for example  Theorem IV.6.7 in \cite{Dunford1958}):
\begin{theorem}
Let $X$ be a compact Hausdorff space. A subset $F\subset C(X;\mathbb{R})$ is relatively compact in the topology induced by uniform convergence, if and only if it is equicontinuous and pointwise bounded.
\end{theorem}

The next proposition follows from the Arzel\'{a}-Ascoli Theorem and will be important for our analysis:
\begin{proposition}\label{C: Application of Arzela Ascoli}
The following holds:
\begin{itemize}
\item[(i)] Let $\mathcal C \subset H$ be a compact set. Then
\begin{equation}\label{Arzela Ascoli deterministic convergence}
\sup_{h\in \mathcal C}\sup_{x\in [0,\Delta_n]}\|(I-S(x))h\|_H\to 0, \quad \text{ as } n\to \infty.
\end{equation}
\item[(ii)] If $\sigma\in L^p(\Omega;L(U,H))$ for some $p\in[1,\infty)$ is an almost surely compact random operator, we get that 
\begin{equation}\label{Arzela Ascoli random operator convergence}
\sup_{x\in [0,\Delta_n]}\|(I-S(x))\sigma\|_{op}\to 0, \quad \text{ as } n\to \infty,
\end{equation}
where the convergence holds almost surely and in $L^p(\Omega;\mathbb{R})$.
\item[(iii)] Let $(\sigma_s)_{s\in [0,T]}$ in $L^p( \Omega\times[0,T];L(U,H))$ for some $p\in[1,\infty)$ be a stochastic process, such that $\sigma_s$ is almost surely compact for all $s\in [0,t]$. If in addition the volatility process is continuous in the $p$'th mean, we obtain
\begin{equation}\label{Arzela Ascoli random operator process L2 convergence}
\sup_{r\in [0,t]}\mathbb{E}[\sup_{x\in [0,\Delta_n]}\|(I-S(x))\sigma_r\|_{op}^p]\to 0 \quad \text{ as } n\to \infty.
\end{equation}
\end{itemize}
\end{proposition}

%For this purpose, the Arzel\'{a}-Ascoli Theorem  is useful.

\begin{proof}
We want to apply the Arzel\'{a}-Ascoli Theorem for the subset $$F:=\lbrace h\mapsto \sup_{x\in [0,\Delta_n]}\| (I-S(x))h\|_H: n\in \mathbb{N}\rbrace\subset C(\mathcal{C};\mathbb{R}).$$
 It is clear that $F$ is pointwise bounded and the equicontinuity holds, since there is a common Lipschitz-constant (independent of $n$):
 \begin{align*}
& | \sup_{x\in [0,\Delta_n]}\| (I-S(x))h\|_H- \sup_{x\in [0,\Delta_n]}\| (I-S(x))g\|_H|\\
&\qquad\qquad\leq  \sup_{x\in [0,\Delta_n]} \| (I-S(x))(h-g)\|_H\\
&\qquad\qquad\leq   \sup_{x\in [0,\Delta_1]} \| (I-S(x))\|_H \|h-g\|_{H},
 \end{align*}
 for all $g,h\in \mathcal C$.
This implies the relative compactness of $F$ with respect to the sup-norm on $C(\mathcal C;\mathbb{R})$. Therefore, there exists a subsequence such that, for $n\to\infty$, we have
\begin{align*}
\sup_{h\in \mathcal C}\sup_{x\in [0,\Delta_{n_k}]}\|(I-S(x))h\|\to 0.
\end{align*}
 Since the sequence $\sup_{x\in [0,\Delta_{n}]}\|(I-S(x))\cdot\|$ is monotone in $n$, we obtain convergence for the whole sequence. This shows (\ref{Arzela Ascoli deterministic convergence}).

Let $B_0(1):=\lbrace h\in H: \| h\|_H=1\rbrace$ be the unit sphere in $H$ and fix $\omega\in \Omega$, such that $\sigma(\omega)$ is compact. Since $\sigma(\omega)$ is compact, $\mathcal C:=\overline{\sigma(\omega)(B_0(1))}$ is compact in $H$.
The set $F(\omega)$ of functionals of the form
\begin{align*}
f_n:=\sup_{x\in [0,\Delta_n]}\|(I-\mathcal S (x))\cdot \|_{H}: %\overline{\sigma(\omega)(B_0(1))}
\mathcal C\to \mathbb{R}
\end{align*}
forms an equicontinuous and pointwise bounded subset of $C(\mathcal C
;\mathbb{R})$. 
Thus,  by (\ref{Arzela Ascoli deterministic convergence})
\begin{align*}
\sup_{x\in [0,\Delta_n]}\|(I-\mathcal S (x))\sigma(\omega)\|_{op}=& \sup_{x\in [0,\Delta_n]} \sup_{\|h\|=1}\|(I-\mathcal S (x))\sigma(\omega)h\|_{H}\\
\leq&  \sup_{g\in \mathcal C %\overline{\sigma(\omega)(B_0(1))}
} f_n(g)\\
\to & 0, \quad\text{ as } n\to \infty.
\end{align*}
This gives almost sure convergence.
Since the sequence is uniformly bounded by
$
(1+M(T)) \| \sigma\|_{op},
$
which has finite $p$th moment, we obtain $L^p(\Omega;\mathbb R)$-convergence by the dominated convergence theorem, and therefore (\ref{Arzela Ascoli random operator convergence}) holds.

To verify the convergence (\ref{Arzela Ascoli random operator process L2 convergence}) we argue as follows:  
Defining 
$$
g_n(s):=\left(\mathbb E[\sup_{x\in [0,\Delta_n]}\|(I-\mathcal S (x))\sigma_s\|_{op}^p]\right)^{\frac 1p},
$$
we obtain pointwise boundedness with the bound
%{\bf Dennis, you need assumption 2 on $\sigma$ to claim uniform bound on this one, right? Dennis: Yes, true! I think uniform bound is here misleading (I thought of uniform in $n$), so I cancelled it, since we do not need uniform boundednedd, but pointwise for which Assumption 2 is not needed, right?}
$(1+M(T)) \mathbb E [\|\sigma_s\|_{op}^p]^{\frac 1p}$ and equicontinuity of $\lbrace g_n:n\in\mathbb N\rbrace\subset C([0,t];\mathbb{R})$ by the continuity in the $p$th mean of the process $(\sigma_s)_{s\in[0,T]}$, since by the Minkowski inequality
\begin{align*}
%& \left|\left(\EE
%     \left[\sup_{x\in[0,\Delta_n]}\left\Vert(\mathcal (I-\mathcal S (x)) \sigma_t) \right\Vert_{\text{op}}^p\right]\right)^{\frac 1p}-\left(\EE
%     \left[\sup_{x\in[0,\Delta_n]}\left\Vert(\mathcal (I-\mathcal S (x)) \sigma_s) \right\Vert_{\text{op}}^p\right]\right)^{\frac 1p}\right|\\
\vert g_n(t)-g_n(s)\vert     &\leq \left(\mathbb E \left[\sup_{x\in[0,\Delta_n]}\|(I-\mathcal S (x))(\sigma_t-\sigma_s)\|_{op}^p\right]\right)^{\frac 1p}\\
     &\leq (I+M(T)) \left(\mathbb E \left[\|\sigma_t-\sigma_s\|_{op}^p\right]\right)^{\frac 1p}.
\end{align*}
By the Arzel\'{a}-Ascoli Theorem this induces the convergence of a subsequence of $(b_n)_{n\in\mathbb N}$ in the $\sup$-norm and thus since $b_n$ decreases pointwise with $n$, the convergence of the whole sequence. 
For all $s\in [0,T]$ we have by (\ref{Arzela Ascoli random operator convergence}) that $(\sup_{x\in [0,\Delta_n]}\|(I-\mathcal S (x))\sigma_s\|_{op})$ goes to zero as $n\to \infty$ almost surely. By uniqueness of the limit (in probability), this implies that $b_n(s)$ converges to zero and thus, $\sup_{s\in[0,T]}b_n(s)$ goes to zero.
%Since this is nothing else than the uniform convergence in $p$'th mean of the sequence $(s\mapsto a_n(s))_{n\in\mathbb N}$, which converges almost surely to $0$ by (\ref{Arzela Ascoli random operator convergence})  pointwise in $s$, we obtain the assertion by the uniqueness of the limit.
\end{proof}

Recall also the following fact:
\begin{lemma}
The family $(\mathcal{S}(t)^*)_{t\geq0}$ of adjoint operators of the $C_0$-semigroup $(\mathcal{S}(t))_{t\geq0}$ forms again a $C_0$-semigroup on $H$.
\end{lemma}
\begin{proof}
See Section 5.14 in \cite{Engel1999}.
\end{proof}

Now we can proceed with the proof of our main theorem in the next subsection.

\subsubsection{Proof of Theorem \ref{T: LLN for semigroup case}}
The operator bracket process for the semigroup-adjusted increment takes the form
\begin{equation}
\label{eq:variation-increments}
\langle\langle \widetilde{\Delta}_n^iY\rangle\rangle=\int_{t_{i-1}}^{t_i}\mathcal S(t_i-s)\sigma_sQ\sigma_s^*\mathcal S(t_i-s)^*ds.
\end{equation}
\iffalse We will establish conditions under which this convergence holds, by investigating the following stronger\textbf{???} convergences separately:
\begin{equation}\label{First Convergence}
 \lim_{n\to\infty}\mathbb{E}[\sup_{0\leq t\leq T}\|   \sum_{i=1}^{\ul}\langle\langle \tilde{\Delta}_n^iY\rangle\rangle 
  -\int_0^t\sigma_sQ\sigma_s^*ds \|_{\text{HS}}] 
=0
\end{equation}
as well as
\begin{equation}\label{Second Convergence}
\lim_{n\to\infty}\mathbb{E}[\sup_{0\leq t\leq T}\|    \sum_{i=1}^{\ul}[(\tilde{\Delta}_n^iY)^{\otimes 2}-\langle\langle \tilde{\Delta}_n^iY\rangle\rangle\|_{\text{HS}}] 
= 0
, \qquad \text{ as } n \to \infty,
\end{equation}
where both of the convergences will be proven in $L^1(\Omega;L_{HS}(H))$ with respect to the Hilbert-Schmidt norm.

 We are going to establish conditions, which guarantee the convergence
\begin{equation}
    \sum_{i=1}^{\ul}[(\tilde{\Delta}_n^iY)^{\otimes 2}-\langle\langle \tilde{\Delta}_n^iY\rangle\rangle] 
    \stackrel{L^1}{\rightarrow} 0
, \qquad \text{ as } n \to \infty.
\end{equation}\fi
For $i\in \{1, \dots, \ul\}$ we denote by  $
\Delta_n^iW:=W_{t_i}-W_{t_{i-1}}$ and:
\begin{align*}
    \tilde{\beta}_i^n&:=\mathcal S (t_i-t_{i-1})\sigma_{t_{i-1}}\Delta_n^i W,\\
    \tilde{\chi}_i^n&:=\int_{t_{i-1}}^{t_i}[\mathcal S (t_i-s)\sigma_s-\mathcal S (t_i-t_{i-1})\sigma_{t_{i-1}}]dW_s.
\end{align*}
Then 
\begin{align*}
\tilde{\Delta}_n^iY&=  \tilde{\beta}_i^n +\tilde{\chi}_i^n.
\end{align*}
%\subsubsection*{Strategy of the proof}

To this end, fix some $T>0$. Using the triangle inequality, we can estimate
\begin{align}\nonumber
&\sup_{t\in[0,T]}\left\Vert\sum_{i=1}^{\ul}(\tilde{\Delta}_n^iY)^{\otimes 2}
-\int_0^t\sigma_sQ\sigma_s^*ds \right \Vert_{\text{HS}}\\\label{eq:component1} 
  &\leq  \sup_{t\in[0,T]}\left\Vert\sum_{i=1}^{\ul}(\tilde{\Delta}_n^iY)^{\otimes 2}
      -(\tilde{\beta}_i^n)^{\otimes 2}\right \Vert_{\text{HS}}
    \\\label{eq:component2} 
    &\qquad +\sup_{t\in[0,T]} \left \Vert\sum_{i=1}^{\ul}(\tilde{\beta}_i^n)^{\otimes 2}
  -\mathcal S (t_i-t_{i-1})\sigma_{t_{i-1}}Q\sigma_{t_{i-1}}^*\mathcal S (t_i-t_{i-1})^* \Delta_n \right\Vert_{\text{HS}}\\\label{eq:component3} 
   &\qquad + \sup_{t\in[0,T]}\left\Vert\sum_{i=1}^{\ul}   \int_{(i-1)\Delta_n}^{i\Delta_n}\mathcal S (t_i-t_{i-1})\sigma_{t_{i-1}}Q\sigma_{t_{i-1}}^*\mathcal S (t_i-t_{i-1})^*\notag \right.\\
 &\qquad\qquad\left. -\mathcal S (t_i-s)\sigma_sQ\sigma_s^*\mathcal S(t_i-s)^*ds \right\Vert_{\text{HS}}\\
 &\qquad+\sup_{t\in[0,T]}\left\Vert\sum_{i=1}^{\ul}
 \langle\langle \tilde{\Delta}_n^iY\rangle\rangle-\int_0^t\sigma_sQ\sigma_s^*ds \right \Vert_{\text{HS}}\label{eq:component4}.
 \end{align}

%%%%%%%%%%%%%%%%%%%%%%%%%%%%%%%%%%%%%%%%%%%%%%%
%%%%%%%%%%%%%%%%%%%%%%%%%%%%%%%%%%%%%%%%%%%%%%%%%%
Before we proceed, we need the following result:
\begin{lemma}\label{le:bounds}
Under Assumption \ref{as:fourthmomentvol}, we have
%Assume that $\EE\left[\Vert\sigma_s\Vert_{\text{op}}^4\right]\leq K$ for all $s\geq 0$. Then we have
\begin{align}\label{eq:boundbeta}
&\EE\left[\Vert\tilde{\beta}_i^n\Vert_H\right]\leq M(\Delta_n)\sqrt{\text{Tr}(Q) \sqrt{C_2(T)}} \Delta_n^{1/2}, \\
&\EE\left[\Vert\tilde{\beta}_i^n\Vert_H^2\right]\leq M(\Delta_n)^2\text{Tr}(Q) \sqrt{C_2(T)} \Delta_n,\\
&\EE\left[\Vert\tilde{\beta}_i^n\Vert_H^4\right]\leq M(\Delta_n)^4(\text{Tr}(Q)+2\text{Tr}_2(Q)) C_2(T) \Delta_n^2.
  %\EE||\tilde{\beta}_i^n||_H\leq M^p K \Delta_n^{1/2}   \text{ and }
  %\EE||\beta_i^n||_H^2\leq K \Delta_n.
\end{align}
Under Assumptions \ref{as:smoothvol}, \ref{as:fourthmomentvol} and either \ref{as:Q is more than Hilbert Schmidt}(a) or \ref{as:Q is more than Hilbert Schmidt}(b), we have
\begin{align}\label{eq:boundchi}
    & \EE\left[\Vert\tilde{\chi}_i^n\Vert_H^2\right]
 \leq  \Delta_na_n(T),
%\text{ and    }
%&  \EE\left[\Vert\tilde{\chi}_i^n\Vert_H\right]
% \leq   \Delta_n^{\frac12}a_n^{\frac 12}(T)
\end{align}
for some constant $K(T)>0$ and a sequence $(a_n(T))_{n\in\mathbb{N}}$ of real numbers converging to zero.
\end{lemma}
\begin{proof}
First notice that the trace class property of $Q$ yields $\|Q^{1/2}\Vert_{\text{HS}}^2=\text{Tr}(Q)<\infty$. Using the It\^{o} isometry, see 
\citet[Corollary 8.7, p.~123]{PZ2007}, we deduce from Assumption 2 that 
\begin{align*}
    \EE\left[\Vert\tilde{\beta}_i^n\Vert_H^2\right] &= \Delta_n \EE\left[\left\Vert S(t_i-t_{i-1})\sigma_{t_{i-1}}Q^{1/2}\right\Vert_{\text{HS}}^2\right]\\
    &\leq  M(\Delta_n)^2\Delta_n\EE\left[\Vert \sigma_{t_{i-1}}\Vert_{\text{op}}^2\right]\Vert Q^{1/2}\Vert_{\text{HS}}^2 \\
    &\leq M(\Delta_n)^2\text{Tr}(Q) \sqrt{C_2(T)} \Delta_n,
\end{align*}
   where $M(\Delta_n)$ is given by \eqref{Global Bound for the semigroup}. 
An application of the Cauchy-Schwarz inequality gives  
$$
\EE\left[\Vert\tilde{\beta}_i^n\Vert_H\right]\leq\sqrt{\EE\left[\Vert\beta_i^n\Vert_H^2\right]},
$$ 
which leads to the result for $p=1$.

For the fourth moment, we argue as follows: By the independent increment property of $W$, we have that
$\Delta_i^nW$ is independent of the $\mathcal F_{(i-1)\Delta_n}$-measurable random variable $\sigma_{(i-1)\Delta_n}$. Thus, again by using the bound \eqref{Global Bound for the semigroup} on the semigroup  gives
\begin{align*}
   \EE\left[\Vert\tilde{\beta}_i^n\Vert_H^4\right]&\leq M(\Delta_n)^4 \EE\left[\Vert\sigma_{t_{i-1}\Delta_n}\Vert_{\text{op}}^4\Vert\Delta_i^nW\Vert_H^4\right]
   \\
   %&=M^4\EE\left(\EE\left(\left.
   %||\sigma_{t_{i-1}}||_{L(U,H)}^4||\Delta_i^nW||_H^4\right|\mathcal{F}_{(i-1)\Delta_n}\right)\right)
   %\\
   &=M(\Delta_n)^4\EE\left[\Vert\sigma_{t_{i-1}}\Vert_{\text{op}}^4\right]\EE\left[\Vert\Delta_i^nW\Vert_H^4\right] \\
   &\leq M(\Delta_n)^4C_2(T) \left(\text{Tr}(Q)^2+2\text{Tr}_2(Q)\right)\Delta_n^2,
\end{align*}
after appealing to Lemma \ref{lemma:4thmoment} and Assumption \ref{as:fourthmomentvol}.% where $K(T)$ is taken to be greater than $M^4 C_2(T)\mu_{\Delta}^4$ for $\mu_{\Delta}^4$ denotes the fourth moment of the Gaussian with variance $\Delta_n$ (which shrinks with $\Delta_n$). {\bf Fred: this here does not make sense, because choosing $K$ as claimed gives no $\Delta_n$ in the estimate! I think that $K(T)$ is chosen so that it is greater than $M
%^4C_2(T)\mu^4$, where $\mu^4$ is the fourth moment of a standard Gaussian (variance equal to 1)! $\mu$ is then 3, right? Dennis: Use Lemma 2.7!!!}

We have, by Assumption \ref{as:smoothvol}, that
\begin{align*}
  & \sup_{s\in (t_{i-1},t_i]} \EE\left[\Vert(\sigma_s-\sigma_{t_{i-1}})Q^{1/2} \Vert_{\text{HS}}^2\right]
 \leq C_1^2(T) \Delta_n^{2\alpha}.
\end{align*}
Hence, for all $i\in\{1, \dots, \ul\}$
\begin{align}\label{Zweigstelle1}
   \int_{t_{i-1}}^{t_i} \EE\left[\Vert(\sigma_s-\sigma_{(i-1)\Delta_n})Q^{1/2} \Vert_{\text{HS}}^2\right] ds \leq  C_1^2(T)\Delta_n^{1+2\alpha}.
\end{align}
By the It\^{o} isometry
\begin{align}\label{Zweigstelle2}
\EE \left[\| \tilde{\chi}_i^n\|_H^2\right]=& \int_{t_{i-1}}^{t_i} \mathbb{E}\left[\|(\mathcal S (t_i-s) \sigma_s-\mathcal S (t_i-t_{i-1})\sigma_{t_{i-1}})Q^{\frac 12}\|^2_{HS}\right]ds\\\notag
\leq &  \int_{t_{i-1}}^{t_i} \mathbb{E}\left[M(\Delta_n)^2\|(\sigma_s-\mathcal S (s-t_{i-1})\sigma_{t_{i-1}})Q^{\frac{1}{2}}\|_{\text{HS}}^2\right]ds\\\notag
\leq &  2M(\Delta_n)^2\int_{t_{i-1}}^{t_i} \mathbb{E}\left[\|(\sigma_s-\sigma_{t_{i-1}})Q^{\frac{1}{2}}\|_{\text{HS}}^2+\|(\mathcal S (s-t_{i-1})\sigma_{t_{i-1}}-\sigma_{t_{i-1}})Q^{\frac{1}{2}}\|_{\text{HS}}^2\right]ds,
\end{align}
where we used the fact that $ \mathcal S (t_i-t_{i-1})=\mathcal S (t_i-s)\mathcal S (s-t_{i-1})$ in the first inequality.

Assume now  Assumption \ref{as:Q is more than Hilbert Schmidt}(a) holds and 
denote by $\sigma_sQ^{\frac 12}=\mathcal{K}_s\mathcal T$ the corresponding decomposition. We obtain
\begin{align*}
\EE \left[\| \tilde{\chi}_i^n\|_H^2\right]&\leq 2M(\Delta_n)^2\int_{t_{i-1}}^{t_i}  \mathbb{E}\left[\|(\mathcal S (s-t_{i-1})-I)\mathcal{K}_{t_{i-1}}\|_{op}^2\right]\|\mathcal{T} \|_{\text{HS}}^2\\
&\qquad +\EE\left[\|(\sigma_s-\sigma_{t_{i-1}} )Q^{\frac{1}{2}}\|_{HS}^2\right]ds\\
&\leq 2M(\Delta_n)^2\left( \Delta_n \mathbb{E}\left[\sup_{x\in [0,\Delta_n]}\|(I-\mathcal S (x))\mathcal{K}_{t_{i-1}}\|_{op}^2\right] \|\mathcal{T} \|_{\text{HS}}^2
+C_1^2(T) \Delta_n^{1+2\alpha}\right).
 \end{align*}
The assertion follows 
 with $$a_n(T)=2M(\Delta_n)^2\left(\sup_{s\in[0,T]}\mathbb{E}\left[\sup_{x\in [0,\Delta_n]}\|(I-\mathcal S (x))\mathcal{K}_{s}\|_{op}^2\right] \|\mathcal{T} \|_{\text{HS}}^2
+C_1^2(T) \Delta_n^{2\alpha}\right),$$
 by \eqref{Arzela Ascoli random operator process L2 convergence} in Corollary \ref{C: Application of Arzela Ascoli}, since $(\mathcal{K}_{s})_{s\in[0,T]}$ is mean square continuous and $\mathcal{K}_s$ is almost surely a compact operator for all $s\in [0,T]$. %Applying also the Cauchy-Schwarz inequality this gives (\ref{eq:boundchi}).

Assume now Assumption \ref{as:Q is more than Hilbert Schmidt}(b) holds.
By (\ref{Zweigstelle2}) and (\ref{Zweigstelle1}) and Assumption \ref{as:fourthmomentvol} we obtain 
\begin{align*}
&\EE \left[\| \tilde{\chi}_i^n\|_H^2\right] \\
&\leq   2M(\Delta_n)^2\int_{t_{i-1}}^{t_i} \mathbb{E}\left[\|(\sigma_s-\sigma_{t_{i-1}})Q^{\frac{1}{2}}\|_{\text{HS}}^2+\|(\mathcal S (s-t_{i-1})\sigma_{t_{i-1}}-\sigma_{t_{i-1}})Q^{\frac{1}{2}}\|_{\text{HS}}^2\right]ds\\
&\leq 2M(\Delta_n)^2 (\int_{t_{i-1}}^{t_i}\mathbb{E}\left[ \|(\sigma_s-\sigma_{t_{i-1}})Q^{\frac{1}{2}}\|_{\text{HS}}^2  +\sup_{r\in[0,\Delta_n]} \|(\mathcal S (r)-I\|_{op}^2\| \sigma_{t_{i-1}}Q^{\frac{1}{2}}\|_{\text{HS}}^2\right]ds)\\
&\leq 2M(\Delta_n)^2 \left(C_1^2(T) \Delta_n^{1+2\alpha}+\Delta_n \sup_{r\in[0,\Delta_n]} \|(\mathcal S (r)-I\|_{op}^2 \sqrt{C_2(T)} \text{Tr}(Q)\right).
\end{align*}
This shows the assertion with 
$$a_n(T)=2M(\Delta_n)^2 \left( \sup_{r\in[0,\Delta_n]} \|(\mathcal S (r)-I\|_{op}^2 \sqrt{C_2(T)} \text{Tr}(Q)+C_1^2(T) \Delta_n^{2\alpha}\right),$$
since, by the uniform continuity of the semigroup, $ \sup_{r\in[0,\Delta_n]} \|(\mathcal S (r)-I\|_{op}$ converges to zero as $n\to\infty$.
\end{proof}

\begin{remark}
In the following, we need Assumption \ref{as:Q is more than Hilbert Schmidt} only if we want to apply Lemma \ref{le:bounds}, where we needed it to verify that the sequence $a_n$ converges to zero.
The convergence rate of $a_n$ is determined by both, the path-regularity of the volatility process as well as the convergence rate of the semigroup (on compacts) as $t\to 0$. The convergence speed of this sequence will essentially determine the rate of convergence of the sequence $b_n$ from Theorem \ref{T: LLN for semigroup case}.
%{\bf Fred: I suggest we also add a remark that the convergence of the $a_n$'s is a mix of the path-regularity of $\sigma$ and the semigroup convergence property for small increments applied to "compacts".}
\end{remark}

\begin{remark}We notice that for the first and second moment estimates of $\Vert\tilde{\beta}_i^n\Vert_H$, we could relax the assumption on $\sigma$ slightly by assuming $\Vert\sigma_s Q^{1/2}\Vert_{\text{HS}}$ having finite second moment. However, the fourth moment of $\Vert\tilde{\beta}_i^n\Vert_H$ is most conveniently estimated based 
on a fourth moment condition on the operator norm of $\sigma$. %{\bf Fred: Dennis, does this remark have something to do with your footnote 1 on page 6? Dennis: I think yes: Should we modify the remark or use operator norm in the assumption?} 
\end{remark}

With the results in Lemma \ref{le:bounds} at hand, we prove convergence of the four components \eqref{eq:component1}-\eqref{eq:component4}.
First, we show the convergence of \eqref{eq:component1}.
\begin{proposition}
Under Assumptions \ref{as:smoothvol}, \ref{as:fourthmomentvol} and \ref{as:Q is more than Hilbert Schmidt}, we have
\begin{align*} \lim_{n\rightarrow\infty}\EE\left[\sup_{t\in[0,T]}  
\left\Vert\sum_{i=1}^{\ul}\left[(\tilde{\Delta}_n^iY)^{\otimes 2}
      -(\tilde{\beta}_i^n)^{\otimes 2}\right]  \right\Vert_{\text{HS}}\right]=0.
\end{align*}
%{\bf We get a rate for the convergence, is that of interest?}
\end{proposition}
\begin{proof}
Define
\begin{align*}
    \tilde{\xi}_i^n&:=(\tilde{\Delta}_n^i Y)^{\otimes 2}-(\tilde{\beta}_i^n)^{\otimes 2} =(\tilde{\beta}_i^n +\tilde{\chi}_i^n)^{\otimes 2}-(\tilde{\beta}_i^n)^{\otimes 2}
    \\
    &=(\tilde{\chi}_i^n)^{\otimes 2}+ \tilde{ \beta}_i^n\otimes \tilde{\chi}_i^n + \tilde{\chi}_i^n \otimes \tilde{\beta}_i^n.
  %  \\
   % &=\left(\int_{(i-1)\Delta_n}^{i\Delta_n}(\sigma_s-\sigma_{(i-1)\Delta_n})dW_s\right)^{\otimes 2} +\beta_i^n \otimes\int_{(i-1)\Delta_n}^{i\Delta_n}(\sigma_s-\sigma_{(i-1)\Delta_n})dW_s \\
%    & \qquad + \int_{(i-1)\Delta_n}^{i\Delta_n}(\sigma_s-\sigma_{(i-1)\Delta_n})dW_s\otimes \beta_i^n.
\end{align*}
By the triangle inequality, we note that 
\begin{align}\nonumber
  \Vert \tilde{\xi}_i^n \Vert_{\text{HS}} 
  &\leq \Vert (\tilde{\chi}_i^n)^{\otimes 2}\Vert_{\text{HS}} 
  +\Vert \tilde{\beta}_i^n\otimes \tilde{\chi}_i^n\Vert_{\text{HS}} 
  +\Vert\tilde{\chi}_i^n \otimes \tilde{\beta}_i^n\Vert_{\text{HS}} 
  \\ \label{eq:xis}
  &= \Vert \tilde{\chi}_i^n\Vert_{H}^2 
  +2\Vert \tilde{\beta}_i^n\Vert_H\Vert \tilde{\chi}_i^n\Vert_{H}. 
\end{align}
Again appealing to the triangle inequality, it follows
\begin{align*}
  \sup_{t\in[0,T]} \left \Vert \sum_{i=1}^{\ul} \tilde{\xi}_i^n\right \Vert_{\text{HS}} \leq \sup_{t\in[0,T]}\sum_{i=1}^{\ul} \Vert \tilde{\xi}_i^n\Vert_{\text{HS}}\leq\sum_{i=1}^{\lfloor T/\Delta_n\rfloor} \Vert \tilde{\xi}_i^n\Vert_{\text{HS}}.
\end{align*}
Applying \eqref{eq:boundchi} in Lemma \ref{le:bounds} leads to
\begin{align*}
   \EE\left[\Vert \tilde{\chi}_i^n\Vert^2_H\right]
    \leq \Delta_n a_n(T).
\end{align*}
We next apply the Cauchy-Schwarz inequality to obtain, using the notation $K_n(T)= M(\Delta_n)^2\text{Tr}(Q) \sqrt{C_2(T)}$,
\begin{align*}
    \EE\left[\Vert\tilde{\beta}_i^n \Vert_{H}
    \Vert\tilde{\chi}_i^n\Vert_{H}\right] ^2
    &\leq \EE\left[\Vert\tilde{\beta}_i^n\Vert_{H}^2\right] \EE\left[\Vert\tilde{\chi}_i^n %\int_{(i-1)\Delta_n}^{i\Delta_n}(\sigma_s-\sigma_{(i-1)\Delta_n})dW_s
    \Vert_{H}^2\right]   \leq K_n(T) \Delta_n^{2} a_n(T),
\end{align*}
by \eqref{eq:boundbeta} and \eqref{eq:boundchi} in Lemma \ref{le:bounds}.
Altogether we have, since $a_n\to 0$ as $n\to \infty$, that
\begin{align}\label{convergence inequality for first summand}
   \EE \left[ \sup_{t\in[0,T]} \left\Vert \sum_{i=1}^{\ul} \tilde{\xi}_i^n\right\Vert_{\text{HS}}\right]
     \leq & \lfloor T/\Delta_n\rfloor (  \Delta_n a_n(T)+2 \sqrt{K_n(T)a_n(T)}\Delta_n),%\to 0, %\text{ as } n \to \infty.
\end{align}
converges to zero as $n\to \infty$ by Lemma \ref{le:bounds}.
\end{proof}
%%%%%%%%%%%%%%%%%%%%%%%%%%%%%%%%%%%%%%%%%%%%%%%%%%%
%End component 1
%%%%%%%%%%%%%%%%%%%%%%%%%%%%%%%%%%%%%%%%%%%%%%%%%%%%%%%%%
Now we prove the convergence of \eqref{eq:component2}.

\begin{proposition}
Under Assumptions \ref{as:smoothvol}, \ref{as:fourthmomentvol} and \ref{as:Q is more than Hilbert Schmidt} we have,
%Assume square integrability of $ \sum_{i=1}^{\ul}(\sigma_{(i-1)\Delta_n}\Delta_i^n W)^{\otimes 2}$. {\bf This condition should be rephrased into a fourth moment condition on $\sigma$!}
\begin{align*}
    \lim_{n\rightarrow\infty}\EE \left[\sup_{t\in[0,T]}  \left\Vert\sum_{i=1}^{\ul}\left\{(\tilde{\beta}_i^n)^{\otimes 2}- \mathcal S(t_i-t_{i-1})\sigma_{t_{i-1}}Q\sigma_{t_{i-1}}^*\mathcal S(t_i-t_{i-1})^*  \Delta_n\right\} \right\Vert_{\text{HS}}^2\right]=0. 
\end{align*}
\end{proposition}
\begin{proof}
We define
\begin{align*}
    \tilde{\zeta}_i^n :=(\tilde{\beta}_i^n)^{\otimes 2}- \mathcal S(t_i-t_{i-1})\sigma_{t_{i-1}}Q\sigma_{t_{i-1}}^*\mathcal S (t_i-t_{i-1})^* \Delta_n.
\end{align*}
First we show that $\sup_{t\in[0,T]}\Vert\sum_{i=1}^{\ul}\tilde{\zeta}_i^n\Vert_{\text{HS}}$ has finite second moment. By the triangle inequality and
 Lemma \ref{lem:HS-banachalg} 
\begin{align*}
\sup_{t\in[0,T]} \left \Vert\sum_{i=1}^{\ul}\tilde{\zeta}_i^n\right \Vert_{\text{HS}}&\leq
\sum_{i=1}^{\lfloor T/\Delta_n\rfloor}\Vert\tilde{\zeta}_i^n\Vert_{\text{HS}} \\
&\leq\sum_{i=1}^{\lfloor T/\Delta_n\rfloor}\Vert(\tilde{\beta}_i^n)^{\otimes 2}\Vert_{\text{HS}}\\
 &\qquad +
\Delta_n\sum_{i=1}^{\lfloor T/\Delta_n\rfloor}\Vert \mathcal S(t_i-t_{i-1})\sigma_{t_{i-1}}Q\sigma_{t_{i-1}}^*\mathcal S (t_i-t_{i-1})^*\Vert_{\text{HS}} \\
&\leq \sum_{i=1}^{\lfloor T/\Delta_n\rfloor}\Vert\tilde{\beta}_i^n\Vert_H^2+\Delta_n\sum_{i=1}^{\lfloor T/\Delta_n\rfloor}\Vert\mathcal S (t_i-t_{i-1})\sigma_{t_{i-1}}Q^{1/2}\Vert^2_{\text{HS}} \\
&\leq \sum_{i=1}^{\lfloor T/\Delta_n\rfloor}\Vert\tilde{\beta}_i^n\Vert_H^2+\Delta_n\text{Tr}(Q)M(\Delta_n)^2\sum_{i=1}^{\lfloor T/\Delta_n\rfloor}\Vert\mathcal \sigma_{t_{i-1}}\Vert^2_{\text{op}}.
\end{align*}
Considering $\EE\left[\sup_{t\in[0,T]}\Vert\sum_{i=1}^{\ul}\tilde{\zeta}_i^n\Vert_{\text{HS}}^2\right]$, we get a finite sum of terms of the type 
$\EE\left[\Vert\tilde{\beta}_i^n\Vert_H^4\right]$, $\EE\left[\Vert\mathcal \sigma_{t_{i-1}}\Vert_{\text{op}}^4\right]$
and $\EE\left[\Vert\tilde{\beta}_i^n\Vert_H^2\Vert\sigma_{t_{i-1}}\Vert^2_{\text{op}}\right]$. The first is finite due to Lemma
\ref{le:bounds}, while the second is finite by the imposed Assumption \ref{as:fourthmomentvol}. For the third, we apply the Cauchy-Schwarz inequality and argue as for the first two. In conclusion, we obtain a finite second moment as desired.  

Note that $R_t=\int_0^t h_s dW(s)$ where $h_s=\sum_{i=1}^n S(t_i-t_{i-1})\sigma_{t_{i-1}} \mathbf{1}_{(t_{i-1},t_i]}(s)$ defines a martingale, such that $R_{t_m}=\sum_{j=1}^{m}\tilde{\beta}_j^n$. Then the squared process is
\begin{align*}
\int_0^{t_m} h_s dW(s)^{\otimes 2} =\sum_{i,j=1}^m\langle \tilde{\beta}_i^n,\cdot\rangle \tilde{\beta}_j^n
\end{align*}
and
\begin{align*}
\langle\langle &\int_0^{\cdot} h_s dW(s)\rangle\rangle_{t_m} \\
&\qquad=\int_0^{t_m} \sum_{i,j=1}^m\mathcal{S}(t_i-t_{i-1})\sigma_{t_{i-1}}Q\sigma_{t_{j-1}}^*\mathcal{S}(t_j-t_{j-1})^*\mathbf{1}_{[t_{i-1},t_i)}(s)\mathbf{1}_{[t_{j-1},t_j)}(s)ds\\
&\qquad=\int_0^{t_m} \sum_{i=1}^m\mathcal{S}(t_i-t_{i-1})\sigma_{t_{i-1}}Q\sigma_{t_{i-1}}^*\mathcal{S}(t_i-t_{i-1})^* \mathbf{1}_{[t_{i-1},t_i)}(s)ds.
\end{align*}
We obtain  that
\begin{align*}
\tilde{\zeta}_m^n&=\int_0^{t_m} h_s dW(s)^{\otimes 2}-\langle\langle \int_0^{\cdot} h_s dW(s)\rangle\rangle_{t_m}\\
&\qquad\qquad-\int_0^{t_{m-1}} h_s dW(s)^{\otimes 2}+\langle\langle \int_0^{\cdot} h_s dW(s)\rangle\rangle_{t_{m-1}}
\end{align*}
 forms a sequence of martingale differences with respect to $(\mathcal{F}_{t_{i-1}})_{i\in \N}$, by Remark \ref{rem:martingale}.
 This implies in particular, after double conditioning, that for $1\leq i\neq j\leq \ul $, 
 $$
 \mathbb{E}\left[\langle\tilde{\zeta}_i^n, \tilde{\zeta}_j^n\rangle_{\text{HS}}\right]=0.
 $$
 %Hence {\bf Dennis, below you have $t$ in $\ul$ on LHS, but $\lfloor T/\Delta_n\rfloor$ on RHS? Should there be a $t$ on RHS as well, or is there a Doob inequality argument going on here? We need a sup inside the expectation on the RHS eventually! Dennis: Yes there was amistake, I included a Doob argument. However, is the seemingly argument at the same place from the old document right? I do not understand it there, but if it is right, it seems to be easier...} 
 By Doob's martingale inequality we obtain
 \begin{align*}
  \EE  \left[\sup_{t\in[0,T]}\left\Vert\sum_{i=1}^{\ul}\tilde{\zeta}_i^n \right\Vert_{\text{HS}}^2 \right]\leq 4  \EE  \left[\left\Vert\sum_{i=1}^{\lfloor T/\Delta_n\rfloor}\tilde{\zeta}_i^n \right\Vert_{\text{HS}}^2 \right]= 4 \sum_{i=1}^{\lfloor T/\Delta_n\rfloor} \EE\left[\Vert\tilde{\zeta}_i^n\Vert_{\text{HS}}^2\right].
\end{align*}
Applying the  triangle inequality and the basic inequality $(a+b)^2\leq 2(a^2+b^2)$, we find
\begin{align*}
    \Vert \tilde{\zeta}_i^n\Vert_{\text{HS}}^2&\leq 2\left(\Vert(\tilde{\beta}_i^n)^{\otimes 2}\Vert_{\text{HS}}^2 + \Vert \mathcal S (t_i-t_{i-1})\sigma_{t_{i-1}}Q\sigma_{t_{i-1}}^*\mathcal S (t_i-t_{i-1})^*  \Vert_{\text{HS}}^2  \Delta_n^2 \right)\\
    &\leq2\left(\Vert \tilde{\beta}_i^n\Vert_{H}^4 + \Vert \sigma_{t_{i-1}}Q\sigma_{t_{i-1}}^* \Vert_{\text{HS}}^2 M(\Delta_n)^4 \Delta_n^2 \right).
\end{align*}
Denoting again $K_n(T)= M(\Delta_n)^2\text{Tr}(Q) \sqrt{C_2(T)}$, we can now apply 
Lemma \ref{le:bounds} to conclude that
\begin{align}\label{convergence inequality for second summand}
  \sum_{i=1}^{\lfloor T/\Delta_n\rfloor} \EE\left[\Vert\tilde{\zeta}_i^n\Vert_{\text{HS}}^2\right]
  &\leq 2\left(K_n(T)\lfloor T/\Delta_n\rfloor \Delta_n^2+M(\Delta_n)^4\Delta_n\EE\left[ \sum_{i=1}^{\lfloor T/\Delta_n\rfloor} \Vert \sigma_{t_{i-1}}Q\sigma_{t_{i-1}}^*  \Vert_{\text{HS}}^2 \Delta_n \right]\right)\\
  &\to 0, \text{ as } n \to \infty,\notag 
\end{align}
since the expectation operator on the right-hand side of the inequality above converges to 
$$
\EE\left[ \int_0^T\Vert \sigma_{s}Q\sigma_{s}^*  \Vert_{\text{HS}}^2ds\right]<\infty.
$$
Hence, the proposition follows. 
\end{proof}

%%%%%%%%%%%%%%%%%%%%%%%%%%%%%%%%%%%%%%%%%%%%%%%%%%%%%%
%End of component 2
%%%%%%%%%%%%%%%%%%%%%%%%%%%%%%%%%%%%%%%%%%%%%

Next, we prove the convergence of \eqref{eq:component3}.
%To shorten the arguments we will make use of the notation
%$$|A|=(A^*A)^{\frac 12}$$
%for an operator $A\in L(U,H)$.
\begin{proposition}
Assume that Assumptions \ref{as:smoothvol} and \ref{as:fourthmomentvol} hold. 
Then 
\begin{align*}
   \lim_{n\rightarrow\infty}\EE [  \sup_{t\in[0,T]}\Vert \sum_{i=1}^{\ul} \int_{(t_{i-1}}^{t_i} & (\mathcal S (t_i-t_{i-1})\sigma_{t_{i-1}}Q \sigma_{t_{i-1}}^* \mathcal S(t_i-t_{i-1})^* \\
   &-\mathcal S (t_i-s)\sigma_sQ\sigma_s^*\mathcal S (t_i-s)^*ds \Vert_{HS}]=0.
   %, \text{ as } n \to \infty.
\end{align*}
\end{proposition}
\begin{proof}
From the triangle and Bochner inequalities, we get
\begin{align*}
 &\Vert \sum_{i=1}^{\ul} \int_{t_{i-1}}^{t_i}(\mathcal S (t_i-t_{i-1})\sigma_{t_{i-1}}Q \sigma_{t_{i-1}}^* \mathcal S(t_i-t_{i-1})^*\\
   &\qquad-\mathcal S (t_i-s)\sigma_sQ\sigma_s^*\mathcal S (t_i-s)^*ds\Vert_{HS}
    \\
    &\qquad\qquad\leq 
   \sum_{i=1}^{\lfloor T/\Delta_n \rfloor} \int_{t_{i-1}}^{t_i}
     \Vert(\mathcal S (t_i-t_{i-1})\sigma_{t_{i-1}}Q \sigma_{t_{i-1}}^* \mathcal S(t_i-t_{i-1})^*\\
   &\qquad\qquad\qquad-\mathcal S (t_i-s)\sigma_sQ\sigma_s^*\mathcal S (t_i-s)^*\Vert_{\text{HS}}ds.
\end{align*}
Note that for $s\in (t_{i-1},t_i]$, we have 
\begin{align*}
 \mathcal S (t_i-t_{i-1})\sigma_{t_{i-1}}&Q \sigma_{t_{i-1}}^* \mathcal S(t_i-t_{i-1})^*-\mathcal S (t_i-s)\sigma_sQ\sigma_s^*\mathcal S (t_i-s)^*
   \\
   =&(\mathcal S (t_i- t_{i-1})\sigma_{t_{i-1}}-\mathcal S (t_i - s) \sigma_s)Q\sigma_{t_{i-1}}^*\mathcal S (t_i-t_{i-1})^*\\
   & +\mathcal S (t_i -s) \sigma_sQ(\sigma_{t_{i-1}}^*\mathcal{S}(t_i-t_{i-1})^*-\sigma_s^*\mathcal S (t_i-s)^*).
\end{align*}
Hence, using the triangle inequality and then the Cauchy-Schwarz inequality, we have
\begin{align*}
  &\EE \left[\Vert \mathcal S (t_i-t_{i-1})\sigma_{t_{i-1}}Q \sigma_{t_{i-1}}^* \mathcal S(t_i-t_{i-1})^*-\mathcal S (t_i-s)\sigma_sQ\sigma_s^*\mathcal S (t_i-s)^*\Vert_{\text{HS}}\right]^2 \\
   &\qquad= \EE\left[\Vert  
    (\mathcal S (t_i- t_{i-1})\sigma_{t_{i-1}}-\mathcal S (t_i - s) \sigma_s)Q\sigma_{t_{i-1}}^*\mathcal S (t_i-t_{i-1})^*\right.\\
   & \qquad\qquad\left.+\mathcal S (t_i -s) \sigma_sQ(\sigma_{t_{i-1}}^*\mathcal{S}(t_i-t_{i-1})^*-\sigma_s^*\mathcal S (t_i-s)^*)
    \Vert_{\text{HS}}\right]^2 \\
       &\qquad\leq  2 \EE \left[\Vert  
    (\mathcal S (t_i- t_{i-1})\sigma_{t_{i-1}}-\mathcal S (t_i - s) \sigma_s)Q^{\frac 12}\|_{op}\|Q^{\frac 12}\sigma_{t_{i-1}}^*\mathcal S (t_i-t_{i-1})^*\|_{\text{HS}}\right]^2\\
   & \qquad\qquad+2\mathbb E \left[\| \mathcal S (t_i -s) \sigma_sQ^{\frac 12}\|_{\text{HS}}\|Q^{\frac 12}(\sigma_{t_{i-1}}^*\mathcal{S}(t_i-t_{i-1})^*-\sigma_s^*\mathcal S (t_i-s)^*)
    \Vert_{\text{op}}\right]^2\\
        &\qquad\leq  2\EE \left[\Vert  
    (\mathcal S (t_i- t_{i-1})\sigma_{t_{i-1}}-\mathcal S (t_i - s) \sigma_s)Q^{\frac 12}\|_{op}^2\right]\EE\left[\|Q^{\frac 12}\sigma_{t_{i-1}}^*\mathcal S (t_i-t_{i-1})^*\|_{\text{HS}}^2\right]\\
   & \qquad\qquad+2\mathbb E \left[\| \mathcal S (t_i -s) \sigma_sQ^{\frac 12}\|_{\text{HS}}^2\right]\EE\left[\|Q^{\frac 12}(\sigma_{t_{i-1}}^*\mathcal{S}(t_i-t_{i-1})^*-\sigma_s^*\mathcal S (t_i-s)^*)
    \Vert_{\text{op}}^2\right] .
    \end{align*}
 Thus, using the identity      $S(t_i-t_{i-1})=S(t_i-s)S(s-t_{i-1})$, 
% {\bf Fred: This identity is NOT correct! We have the identity $S(t_i-t_{i-1})=S(t_i-s)S(s-t_{i-1})$, which,  after adding and subtracting, I think is what you can use to do the arguments below (at least something similar...) Also, there is a mix-up of $s$ and $r$ below, right?}
we get
    \begin{align*}
  &  \EE \left[\Vert  (\mathcal S (t_i-t_{i-1})\sigma_{t_{i-1}}Q \sigma_{t_{i-1}}^* \mathcal S(t_i-t_{i-1})^*-\mathcal S (t_i-s)\sigma_sQ\sigma_s^*\mathcal S (t_i-s)^*\Vert_{\text{HS}}\right]^2 \\
   &\qquad\leq 2M(\Delta_n)^4 \EE
     \left[\Vert(\mathcal S((s-t_{i-1}) \sigma_{t_{i-1}}- \sigma_s)Q^{1/2} \Vert_{\text{op}}^2\right]  \EE\left[\Vert Q^{1/2}\sigma_{t_{i-1}}^*\mathcal \Vert_{\text{HS}}^2\right]
     \\
     &\qquad\qquad+2M(\Delta_n)^4
      \EE\left[\Vert \sigma_s Q^{1/2} \Vert_{\text{HS}}^2\right]
      \EE \left[\Vert Q^{1/2}(\sigma_{t_{i-1}}^*\mathcal S (s-t_{i-1})^*-\sigma_s^*) \Vert_{\text{op}}^2\right]
      \\
       &\qquad\leq 4 M(\Delta_n)^4 \sup_{r\in[0,T]}  \EE\left[\Vert\sigma_rQ^{\frac 12} \Vert_{\text{HS}}^2\right] \EE
     \left[\sup_{x\in[0,\Delta_n]}\Vert(\mathcal S (x)\sigma_{t_{i-1}}- \sigma_s)Q^{\frac 12} \Vert_{\text{op}}^2\right].
     \end{align*}
By Assumption \ref{as:fourthmomentvol} we know that$$ A_n:=4 M(\Delta_n)^4 \sqrt{C_2(T)}\geq 2 M(\Delta_n)^4 \sup_{r\in[0,T]}   \EE[\Vert\sigma_rQ^{\frac 12} \Vert_{\text{HS}}^2].$$
    %since Assumption \ref{as:smoothvol} induces the continuity of $\sigma$ in mean square by
%     $$\mathbb E [\|\sigma_t-\sigma_s\|_{op}^2]\leq \mathbb E [\|C^1\|_{op}^2]|t-s|^{2\alpha}.$$ 
Using Assumption \ref{as:smoothvol}, this gives the following estimate:
     \begin{align*}
      &\EE \left[\Vert(\mathcal S (t_i-t_{i-1})\sigma_{t_{i-1}}Q \sigma_{t_{i-1}}^* \mathcal S(t_i-t_{i-1})^*-\mathcal S (t_i-s)\sigma_sQ\sigma_s^*\mathcal S (t_i-s)^*\Vert_{\text{HS}}\right]^2 \\
       &\qquad\leq A_n(T) \EE
     \left[\sup_{x\in[0,\Delta_n]}\Vert(\mathcal S (x)\sigma_{t_{i-1}}-\sigma_{t_{i-1}}+\sigma_{t_{i-1}}- \sigma_s)Q^{\frac 12} \Vert_{\text{op}}^2\right]\\
        & \qquad\leq A_n(T)  2\left( \EE
     \left[\sup_{x\in[0,\Delta_n]}\left\Vert(\mathcal S (x)-I) \sigma_{t_{i-1}} Q^{\frac 12} \right\Vert_{\text{op}}^2\right]+\EE
     \left[\left\Vert(\mathcal  \sigma_{t_{i-1}}- \sigma_s)Q^{\frac 12} \right\Vert_{\text{op}}^2\right]\right) \\
         &\qquad \leq A_n(T) 2 (b_n(T)+C_1^2(T)\Delta_n^{2\alpha}),
  %    &\qquad\leq K \Delta_n^{2\alpha}
\end{align*}
where $b_n(T):=\sup_{s\in[0,T]}\EE
     [\sup_{x\in[0,\Delta_n]}\left\Vert(I-\mathcal S (x)) \sigma_{s} Q^{\frac 12} \Vert_{\text{op}}^2\right]$ as before. We have that $(b_n(T))_{n\in\mathbb N}$ is a real sequence converging to 0 by \eqref{Arzela Ascoli random operator process L2 convergence} in Corollary \ref{C: Application of Arzela Ascoli}, since for each $s\in [0,T]$ the operator $\sigma_sQ^{\frac 12}$ is almost surely compact as a Hilbert-Schmidt operator and the process $(\sigma_sQ^{\frac 12})_{s\in [0,T]}$ is mean square continuous by Assumption \ref{as:smoothvol}. 
     
     %{\bf Note that the Corollary assumes some conditions on compactness on $\sigma$, which is NOT assumed in this Propositions to hold. Must impose condition/assumption!!!! Dennis: I thought like follows: $\sigma Q
   % ^{\frac 12}$ is already a Hilbert Schmidt operator and therefore compact. We needed the Assumption 3 just in case that we have to split $Q$ and $\sigma$ in order to obtain an estimate where we get out with the operator norm! I will explain this and hope my reasoning makes sense!}
%In the last inequality we also applied the 
%H\"older continuity of Assumption \ref{as:smoothvol} (and the finite second moment of the H\"older constant $C_1$) and the 
%fourth moment condition in Assumption \ref{as:fourthmomentvol} to bound the second moment of $\Vert\sigma_u\Vert_{\text{op}}$. 
Summing up, we obtain
\begin{align}\label{convergence inequality for third summand}
 &\EE [  \Vert \sum_{i=1}^{\ul} \int_{t_{i-1}}^{t_i}  (\mathcal S (t_i-t_{i-1})\sigma_{t_{i-1}}Q \sigma_{t_{i-1}}^* \mathcal S(t_i-t_{i-1})^*
 \\
 &\qquad-\mathcal S (t_i-s)\sigma_sQ\sigma_s^*\mathcal S (t_i-s)^*ds\Vert_{HS}]\notag
    \\
    &\qquad\qquad\qquad\leq   \sum_{i=1}^{\lfloor T/\Delta_n\rfloor} \int_{t_{i-1}}^{t_i}  (  A_n(T) 2 (C_1^2(T)\Delta_n^{2\alpha}+ b_n(T)))^{\frac 12} ds\notag
    \\
     &\qquad\qquad\qquad=  \lfloor T/\Delta_n\rfloor \Delta_n ( A_n(T) 2 (C_1^2(T)\Delta_n^{2\alpha}+ b_n(T)))^{\frac 12}  \to 0, \text{ as } n \to \infty,\notag
\end{align}
and the proof is complete.
\end{proof}
%%%%%%%%%%%%%%%%%%%%%%%%%End component 3

Finally, we prove the convergence of (\ref{eq:component4}).

\begin{proposition}
Suppose that Assumption \ref{as:smoothvol} and \ref{as:fourthmomentvol} hold.
Then% process $(\Sigma_i \langle\langle \tilde{\Delta}_n^i Y\rangle\rangle)_n$ converges to the operator $\int_0^t\sigma_sQ\sigma_s^*ds$ almost surely and 
$$\lim_{n\to\infty}\mathbb{E}\left[\sup_{0\leq t\leq T}\left\| \sum_{i=1}^{\ul}\langle\langle \tilde{\Delta}_n^iY\rangle\rangle - \int_0^t\sigma_sQ\sigma_s^*ds\right\|_{\text{HS}}\right]=0.$$
\end{proposition}
\begin{proof}
Recall the expression for $\langle\langle \tilde{\Delta}_n^iY\rangle\rangle$ in \eqref{eq:variation-increments}. By the triangle and Bochner inequalities, we find,
\begin{align*}
&\sup_{t\in[0,T]}\left \Vert \int_0^{\ul}\sigma_sQ\sigma_s^*ds-\sum_{i=1}^{\ul}\int_{t_{i-1}}^{t_i}
\mathcal S(t_i-s)\sigma_sQ\sigma_s^*\mathcal S(t_i-s)^*ds\right \Vert_{\text{HS}} \\
&\leq\sup_{t\in[0,T]} \sum_{i=1}^{\ul}\int_{t_{i-1}}^{t_i}\Vert
\sigma_sQ\sigma_s^*-\mathcal S(t_i-s)\sigma_sQ\sigma_s^*\mathcal S(t_i-s)^*\Vert_{\text{HS}}ds \\
&\leq\sum_{i=1}^{\lfloor T/\Delta_n\rfloor}\int_{t_{i-1}}^{t_i}\Vert
\sigma_sQ\sigma_s^*-\mathcal S(t_i-s)\sigma_sQ\sigma_s^*\mathcal S(t_i-s)^*\Vert_{\text{HS}}ds.
\end{align*}
By Lemma \ref{lem:HS-banachalg} and the Cauchy-Schwarz inequality we obtain 
\begin{align*}
&\mathbb{E}\left[\sup_{0\leq t\leq T}\left\| \sum_{i=1}^{\ul}\langle\langle \tilde{\Delta}_n^iY\rangle\rangle - \int_0^t\sigma_sQ\sigma_s^*ds\right\|_{\text{HS}}\right]\\
\leq & \sum_{i=1}^{\lfloor T/\Delta_n\rfloor}\int_{t_{i-1}}^{t_i}\mathbb E [\Vert
(I-\mathcal S(t_i-s))\sigma_sQ\sigma_s^*\Vert_{\text{HS}}]\\
&+\mathbb{E}[\Vert
\mathcal S(t_i-s)\sigma_sQ\sigma_s^*(I-S(t_i-s)^*)\Vert_{\text{HS}}]ds\\
&+ \int_{t_n}^{T}\mathbb{E}[\Vert
\sigma_sQ\sigma_s^*\Vert_{\text{HS}}]ds\\
\leq &\sum_{i=1}^{\lfloor T/\Delta_n\rfloor}\int_{t_{i-1}}^{t_i}\mathbb E[\Vert
(I-\mathcal S(t_i-s))\sigma_sQ^{\frac 12}\|_{\text{op}}\|Q^{\frac 12}\sigma_s^*\Vert_{\text{HS}} ]\\
&+M(\Delta_n)\mathbb E[\Vert
\sigma_sQ^{\frac 12}\|_{\text{HS}}\|Q^{\frac 12}\sigma_s^*(I-S(t_i-s)^*)\Vert_{\text{op}}]ds\\
&+ \int_{t_n}^{T}\mathbb{E}[\Vert
\sigma_s Q^{\frac 12}\Vert_{\text{HS}}^2]ds\\
\leq & \sup_{r\in [0,T]}\mathbb E[ \sup_{x\in [0,\Delta_n]}\Vert
(I-\mathcal S(x))\sigma_rQ^{\frac 12}\|_{op}^2]^{\frac 12}(1+M(\Delta_n)) \int_{0}^{T}\mathbb E[\|Q^{\frac 12}\sigma_s^*\Vert_{\text{HS}}^2]^{\frac 12} ds\\
&+ \int_{t_n}^{T}\mathbb{E}[\Vert
\sigma_sQ^{\frac 12}\Vert_{\text{HS}}^2]ds.
%\to & 0 \quad \text{ as }n\to \infty 
\end{align*}
Using Assumption \ref{as:fourthmomentvol}, we can estimate
\begin{align}\label{convergence inequality for fourth summand}
&\mathbb{E}\left[\sup_{0\leq t\leq T}\left\| \sum_{i=1}^{\ul}\langle\langle \tilde{\Delta}_n^iY\rangle\rangle - \int_0^t\sigma_sQ\sigma_s^*ds\right\|_{\text{HS}}\right]\\
\leq & (b_n(T))^{\frac 12}(1+M(\Delta_n)) T( \sqrt{C_2(T)}\text{Tr}(Q))^{\frac 12}+ (T-t_n)\sqrt{C_2(T)}\text{Tr}(Q)\notag\\
\to & 0 \quad \text{ as }n\to \infty. \notag
\end{align}
Here again $b_n(T):=\sup_{s\in[0,T]}\EE
     [\sup_{x\in[0,\Delta_n]}\left\Vert(I-\mathcal S (x)) \sigma_{s} Q^{\frac 12} \Vert_{\text{op}}^2\right]$, which is a real sequence converging to 0 by \eqref{Arzela Ascoli random operator process L2 convergence} in  Corollary \ref{C: Application of Arzela Ascoli}, since for each $s\in [0,T]$ the operator $\sigma_sQ^{\frac 12}$ is almost surely compact as a Hilbert-Schmidt operator and the process $(\sigma_sQ^{\frac 12})_{s\in [0,T]}$ is mean square continuous by Assumption \ref{as:smoothvol}.
%where the convergence follows by Corollary (\ref{C: Application of Arzela Ascoli}) {\bf Fred: but we need additional assumptions on compactness of $\sigma$ then, which is not assumed here????}, since $(\sigma_sQ^{\frac 12})_{s\in[0,T]}$ is a mean square continuous process of Hilbert-Schmidt, and therefore, compact operators. %Moreover, we have the following bound: 
%$$\|(\Sigma_i \langle\langle \tilde{\Delta}_n^i Y\rangle\rangle)_n\|_{HS}\leq 2M \int_0^{T}\|\sigma_sQ\sigma_s^*\|_{\text{HS}}ds\leq2M \int_0^{T}\|\sigma_sQ^{\frac 12}\|_{\text{HS}}^2ds,$$
%which has finite first moment and gives therefore also the claimed convergence in $L^1(\Omega;L_{HS}(H))$ by the dominated convergence theorem.
\end{proof}

%%%%%%%%%%%%%%%%%%%%%%%%%End component 4

\subsubsection{Proof of Theorem \ref{T: Extension by localization}}

\begin{proof}[Proof of Theorem \ref{T: Extension by localization}]
 Define
 \begin{equation}
 Y^{(m)}_t:=\int_0^t\mathcal S (t-s)\sigma_s^{(m)}dW_s,
\end{equation}
 and
\begin{align*}
\mathcal{Z}^n_m:= &\sup_{0\leq s\leq t}\left\Vert\sum_{i=1}^{\uls}(\tilde{\Delta}_n^i Y^{(m)})^{\otimes 2}-\int_0^s\sigma_u^{(m)} Q\sigma_u^{(m)*}du\right\Vert_{\text{HS}},\\
\mathcal Z^n:= & \sup_{0\leq s\leq t}\left\Vert\sum_{i=1}^{\uls}(\tilde{\Delta}_n^i Y)^{\otimes 2}-\int_0^s\sigma_uQ\sigma_u^*du\right\Vert_{\text{HS}}.
\end{align*}
Since $\sigma^{(m)}$ satisfies the conditions of Theorem \ref{T: LLN for semigroup case}, we obtain that for all $m\in\mathbb N$ and $\epsilon>0$
\begin{equation}\label{Localization Convergence of Zmn}
\lim_{n\to\infty}\mathbb{P}[\mathcal{Z}^n_m>\epsilon]=0.
\end{equation}
%Denote $T_m:=\lbrace \omega\in \Omega: \int_0^T\Vert\sigma_s(\omega)Q^{1/2}\Vert^2_{\text{HS}}ds\leq m\rbrace$. 
We have $\mathcal{Z}_m^n=\mathcal{Z}^n$ on $\Omega_m$ and hence 
\begin{align*}
\mathbb{P}[\mathcal{Z}^n>\epsilon]&= \int_{\Omega_m} \mathbf{1}(\mathcal{Z}^n>\epsilon)d\mathbb P+\int_{\Omega_m^c} \mathbf{1}(\mathcal{Z}^n>\epsilon)d\mathbb P \\
&=\int_{\Omega_m} \mathbf{1}(\mathcal{Z}_m^n>\epsilon)d\mathbb P+\int_{\Omega_m^c} \mathbf{1}(\mathcal{Z}^n>\epsilon)d\mathbb P\\
&\leq \mathbb{P}[\mathcal{Z}_m^n>\epsilon] + \mathbb P[\Omega_m^c],
\end{align*}
which holds for all $n,m\in\mathbb N$.
Now, by virtue of \eqref{Localization Convergence of Zmn} we obtain for all $m\in\mathbb N$ that
$$
\limsup_{n\rightarrow\infty} \mathbb{P}[\mathcal{Z}^n>\epsilon]\leq \mathbb P[\Omega_m^c].
$$
By the continuity of $\mathbb P$ from below, $\mathbb P[\Omega^c_m]$ converges to $0$ as $m\to \infty$ and therefore
$$\lim_{n\rightarrow\infty} \mathbb{P}[\mathcal{Z}^n]=
\limsup_{n\rightarrow\infty} \mathbb{P}[\mathcal{Z}^n>\epsilon]= 0.
$$
\end{proof}

\subsection{Proofs of Section \ref{sec: Applications}}
We will now present the longer proofs of the results presented in Section \ref{sec: Applications}.

\subsubsection{Proof of Theorem \ref{T: Convergence rate for forward curves}}

\begin{proof}[Proof of Theorem \ref{T: Convergence rate for forward curves}]
Since for all $h\in H_{\beta}$ it is $|h(0)|\leq \|h\|_{\beta}$ we have for $\|h\|_{\beta}=1$ that
\begin{align*} 
\| (I-\mathcal S (x)) \sigma_r Q^{\frac 12} h\|_{\beta}\leq & \|(I-\mathcal S (x)) f_r\|_{\beta}+ \left\|(I-\mathcal S (x))\int_0^{\infty} q_r(\cdot,z)h'(z)dz\right\|_{\beta}\\
= & (1)+(2).
\end{align*}
The first summand can be estimated as follows, for some $\zeta\in (0,t)$ and $x<1$:
\begin{align}\label{First step in HJM-convergence speed}
    (1)= & \left(|f_r(x)|^2+\int_0^{\infty} (f_r'(y+x)-f_r'(y))^2e^{\beta y}dy\right)^{\frac 12}\notag\\
   % &\leq (|f(t)|^2+ t^{2\gamma} \|L_1\|_{L^2(\mathbb{R}_+)}^2)^{\frac 12}\\
    &\leq (|f_r'(\zeta)|^2 x^2 + x^{2\gamma} \|L_1\|_{L^2(\mathbb{R}_+)}^2)^{\frac 12}\leq x^{\gamma} (|f_r'(\zeta)|+\|L_r^1\|_{L^2(\mathbb{R}_+)}).
\end{align}
%{\bf Dennis, isn't there a $\Vert h\Vert_{\beta}$ missing above, first term on RHS?}
We can show, using H{\"o}lder inequality, for all $h\in H_{\beta}$ such that $\|h\|_{\beta}=1$, that %{\bf Dennis, I disagree with second equality below. You get $b(x+t,z)$ together with an exponential which also should have a $t$ present? Another minor thing, is it really inequality in the first step?}
\begin{align*}
(2)=&\left(\int_0^{\infty}\left[\partial_y \int_0^{\infty}(q_r(y+x,z)-q_r(y,z))h'(z)dz\right]^2 e^{\beta y}dy\right)^{\frac 12}\\
%=&(\int_0^{\infty}[ \int_0^{\infty}\partial_x(q(x+t,z)-q(x,z))h'(z)dz]^2 e^{\beta x}dx)^{\frac 12}\\
=&\left(\int_0^{\infty}\left[ \int_0^{\infty}\left(e^{-\frac{\beta}{2}x}p_r(y+x,z)-p_r(y,z)\right) e^{\frac{\beta}{2}z-y}h'(z)dz\right]^2 e^{\beta y}dy\right)^{\frac 12}\\
=&\left(\int_0^{\infty}\left[ \int_0^{\infty}(e^{-\frac{\beta}{2}x}p_r(y+x,z)-p_r(y,z)) e^{\frac{\beta}{2}z}h'(z)dz\right]^2 dy\right)^{\frac 12}\\
\leq &\left(\int_0^{\infty} \int_0^{\infty}(e^{-\frac{\beta}{2}x}p_r(y+x,z)-p_r(y,z))^2 dz \|h\|_{\beta} dy\right)^{\frac 12}.
\end{align*}
Now we can estimate, for $x<1$,
\begin{align}\label{Second step in HJM-convergence speed}
(2)\leq &\left(\int_0^{\infty}\int_0^{\infty}(e^{-\frac{\beta}{2}x}(p_r(y+x,z)-p_r(y,z)))^2 dz dy\right)^{\frac 12}\notag\\
&+\left(\int_0^{\infty}\int_0^{\infty} (e^{-\frac{\beta}{2}x}-1)^2 p_r(y,z)^2dxdz\right)^{\frac 12}\notag\\
\leq & x^{\gamma}\|L_r^2\|_{L^2(\mathbb{R}_+^2)}+|e^{-\frac{\beta}{2}x}-1| \|  p_r\|_{L^2(\mathbb{R}_+^2)}\notag\\
\leq & x^{\gamma}\|L_r^2\|_{L^2(\mathbb{R}_+^2)}+  \frac{\beta}{2}   x \|  p_r\|_{L^2(\mathbb{R}_+^2)}\leq x^{\gamma}(\|L_r^2\|_{L^2(\mathbb{R}_+^2)}+  \frac{\beta}{2}    \|  p_r\|_{L^2(\mathbb{R}_+^2)}).
%= & \int_0^{\infty}\int_0^{\infty} (b(x+t,z)-b(x,z))^2dxdz.
\end{align}
Combining \eqref{First step in HJM-convergence speed} and \eqref{Second step in HJM-convergence speed}, we obtain, for $\|h\|_{\beta}= 1$,
\begin{equation}
    \|(I-\mathcal S (x)) \sigma_r Q^{\frac 12} h\|_{\beta}\leq x^{\gamma} [|f_r'(\zeta)|+\|L_r^1\|_{L^2(\mathbb{R}_+)}+ \|L\|_{L^2(\mathbb{R}_+^2)}+  \frac{\beta}{2}    \|  p_r\|_{L^2(\mathbb{R}_+^2)}].
\end{equation}
Now we can conclude that
\begin{align*}
    b_n(T)= & \sup_{r\in[0,T]}\mathbb E[\sup_{x\in [0,\Delta_n]}\sup_{\|h\|_{\beta}=1}\|(I-\mathcal S (x)) \sigma_r Q^{\frac 12} h\|_{\beta}^2]\\
    \leq & \Delta_n^{2\gamma} \sup_{r\in [0,T]}\mathbb E
[(|f_r'(\zeta)|+\|L_r^1\|_{L^2(\mathbb{R}_+)}+ \|L\|_{L^2(\mathbb{R}_+^2)}+  \frac{\beta}{2}    \|  p_r\|_{L^2(\mathbb{R}_+^2)})^2].
\end{align*}
%{\bf Dennis, we need to help the reader a bit. Please refer back to Eq (8) for the $b_n$, and fill in some more details on how you reach the assertion. You sup over $t$ in the semigroup, but in Eq (8) you sup over $x$ instead, so notation is worthwhile to reconsider. You also make a sup over $r$ in Eq (8), where $r$ is time in the volatility. How is this entering the scene above? Finally, maybe it is not good to use notation $b$ for the function, as $b_n$ is something else?}
\end{proof}
%{\bf Dennis, I nowhere see why you assume $h\in C_{loc}^{\gamma}$ in the Theorem. This is a typo, right? If not, then you assume regularity on the noise $W$, which is again depending on $b$, so we go in a dangerous circle... }
%{\bf WHAT DOES THE HÖLDER ASSERTION MEAN? IS IT REASONABLE? PROBABLY IT MEANS THAT $ran(\sigma Q^{\frac 12})\subset C^{1,\gamma}(\mathbb{R}_+)$.
%}

\subsubsection{Proof of Lemma \ref{L: Mean Square Lipschitz continuity of OU-Processes}}
\begin{proof}[Proof of Lemma \ref{L: Mean Square Lipschitz continuity of OU-Processes}]
We have 
\begin{align*}
\Sigma_t-\Sigma_s= & (\mathbb S(t)-\mathbb S(s)) \Sigma_0+\int_s^t \mathbb S(t-u)d\mathcal L_u+ \int_0^s (\mathbb S(t-u)-\mathbb S(s-u))d\mathcal L_u\\:= & (1)+(2)+(3).
\end{align*}
As the semigroup $(\mathbb S(t))_{t\geq 0}$ is uniformly continuous, we can again use the fundamental equality \eqref{Fundamental Theorem of Semigroup Theory II} and the triangle inequality for Bochner integrals to deduce, for $s,t\in [0,T]$ and $t\geq s$, that
$$\|\mathbb{S}(t)-\mathbb S(s)\|_{\text{op}}=\left\|e^{\mathbb{B}s}\int_0^{t-s}   e^{\mathbb B x}\mathbb Bdx\right\|_{\text{op}}=\left\|\int_s^{t}  e^{\mathbb B x}\mathbb Bdx\right\|_{\text{op}}\leq e^{\|\mathbb{B}\|_{\text{op}}T}\|\mathbb B\|_{\text{op}}(t-s).$$
%Observe that by the mean value inequality \textbf{add some source} {\bf Dennis, cannot we just prove this directly as we have the semigroup given as exponential operator?}  the semigroup is Lipschitz continuous as a map from  $[0,T]$ to $L(H)$ with Lipschitz constant $U:=\mathbb C e^{\|\mathbb C \|_{\text{op}}T}$, since
%$$\|S(t)-S(s)\|_{\text{op}}\leq .$$
Denoting $U:=e^{\|\mathbb{B}\|_{\text{op}}T}\|\mathbb B\|_{\text{op}}$, this gives
$$\|(1)\|_{HS}\leq \|\mathbb S(t)-\mathbb S(s)\|_{\text{op}} \|\Sigma_0\|_{\text{HS}}\leq U \|\Sigma_0\|_{\text{HS}} (t-s).$$
This induces
$\mathbb{E}[\|(1)\|_{\text{HS}}^2]^{\frac 12}\leq  U \|\Sigma_0\|_{\text{HS}} (t-s)$.
Moreover, by the It\^{o} isometry
\begin{align*}
\mathbb{E}[\|(2)\|_{\text{HS}}^2]^{\frac 12}=\left(\int_s^t \|\mathbb S(t-u)Q_{\mathcal L}^{\frac 12}\|_{\text{HS}}^2du\right)^{\frac 12}\leq e^{\|\mathbb B\|_{\text{op}}T} \text{Tr}(Q_{\mathcal L})^{\frac 12} (t-s)^{\frac 12},
\end{align*}
where $Q_{\mathcal L}$ denotes the covariance operator of $\mathcal L$.
Finally, we can show again, by the It\^{o} isometry and the mean value inequality, that 
\begin{align*}
\mathbb{E}[\|(3)\|_{\text{HS}}^2]^{\frac 12}= & \left(\int_0^s \|(\mathbb S(t-u)-\mathbb S(s-u))Q_{\mathcal L}^{\frac 12}\|_{\text{HS}}^2du\right)^{\frac 12}\\
\leq &  \left(\int_0^s \|\mathbb S(t-s)-\mathcal{I}\|_{\text{op}}^2\|\mathbb S(s-u))Q_{\mathcal L}^{\frac 12}\|_{\text{HS}}^2du\right)^{\frac 12}\\
\leq &  \left( U^2(t-s)^2 \int_0^s\|\mathbb S(s-u))Q_{\mathcal L}^{\frac 12}\|_{\text{HS}}^2du\right)^{\frac 12}\\
 \leq & U (t-s)  e^{\|\mathbb B\|_{\text{op}}T}  \text{Tr}(Q_{\mathcal L})^{\frac 12}.
\end{align*}
Summing up, we obtain, for $t-s\leq 1$, 
%{\bf Dennis, why the restriction $t-s\leq 1$? Further, why $\sqrt{3}$, you can apply triangle inequality for norms, right, as $\mathbb E[\Vert\cdot\Vert^2]^{1/2}$ is a norm. Dennis: I deleted $\sqrt 3$. The restriction is due to the first estimate of $\mathbb E]\|(1)\|^2]^{\frac 12}\leq ... (t-s)$, such that we can also estimate it with $(t-s)^{\frac 12}$ and the same for the estimate of $\mathbb E]\|(3)\|^2]^{\frac 12}$.}
\begin{align*}
  \mathbb{E}[\| (\Sigma_t-\Sigma_s)\|_{\text{HS}}^2]^{\frac 12} 
\leq &  (\mathbb{E}[\|(1)\|_{\text{HS}}^2]^{\frac 12}+\mathbb{E}[\|(2)\|_{\text{HS}}^2]^{\frac 12}+\mathbb{E}[\|(3)\|_{\text{HS}}^2]^{\frac 12})\\
\leq &  (U \|\Sigma_0\|_{\text{HS}} +e^{\|\mathbb B\|_{\text{op}}T} \text{Tr}(Q_{\mathcal L})^{\frac 12} (1+U)  ) (t-s)^{\frac 12}.
\end{align*}
Since also
  by the It\^{o} isometry, we obtain 
\begin{align*}
\sup_{t \in [0,T] }\mathbb E [\| \Sigma_t^2\|_{\text{HS}}]^{\frac 12}
\leq & \sup_{t \in [0,T] }\left(\|\mathbb S (t) \Sigma_0\|_{\text{HS}}+\mathbb E \left[\left \|\int_0^t \mathbb S(t-u)d\mathcal L_u \right\|_{\text{HS}}^2\right]^{\frac 12}\right) \\
\leq &\sup_{t \in [0,T] }\left(\|\mathbb S (t) \Sigma_0\|_{\text{HS}}+    \left(\int_0^T \|\mathbb S(t-u)Q_{\mathcal L}^{\frac 12}\|_{\text{HS}}^2du\right)^{\frac 12}\right)\\
\leq & e^{\|\mathbb B\|_{\text{op}} T}\|\Sigma_0\|_{\text{HS}}+ e^{\|\mathbb C\|_{\text{op}} T} \text{Tr}(Q_{\mathcal L})^{\frac 12} T^{\frac 12},
\end{align*}
the additional assertion follows by Lemma \ref{L:Squared Volatility Lemma}.
\end{proof}

\section{Discussion and outlook}\label{Sec:Conclusion}
Our paper develops a new asymptotic theory for high-frequency estimation of the volatility of infinite-dimensional stochastic evolution equations in an operator setting. We have defined the so-called semigroup-adjusted realised covariation (SARCV) and derived a weak law of large numbers based on uniform convergence in probability with respect to the Hilbert-Schmidt norm. Moreover, we have presented various examples where our new method is applicable.

Many articles on (high-frequency) estimation for stochastic partial differential equations rely on the so-called spectral approach and assume therefore the applicability of spectral theorems to the generator $A$ (cf. the survey article \cite{Cialenco2018}). This makes it difficult to apply these results on differential operators that do not fall into the symmetric and positive definite scheme, as for instance $A=\frac{d}{dx}$ in the space of forward curves presented in Section \ref{subsect:hjmm}, a case of relevance in financial applications that is included in our framework.   
Moreover, a lot of the related work assumes the volatility as a parameter of estimation to be real-valued (c.f.~the setting in \cite{Cialenco2018}). An exception is the spatio-temporal volatility estimation in the recent paper by \cite{Chong2020}
(see also \cite{ChongDalang2020} for limit laws for the power variation of fractional stochastic parabolic equations). Here, the stochastic integrals are considered in the sense of \cite{Walsh1986} and the generator is the Laplacian. In our analysis, we operate in the general Hilbert space framework in the sense of Peszat and Zabzcyck for stochastic integration and semigroups. 

In our framework, we work with high-frequent observations of Hilbert-space valued random elements, hence we have  observations, which are discrete in time but not necessarily in space.
Recent research on 
 inference for parabolic stochastic partial differential considered observation schemes which allow for discreteness in time and space, cf. \cite{Cialenco2020}, \cite{Bibinger2020}, \cite{Chong2020}, \cite{ChongDalang2020}. However, as our approach falls conveniently into the realm of functional data analysis, we might reconstruct data in several cases corresponding to well-known techniques for interpolation or smoothing.
Indeed, in practice, a typical situation is that the Hilbert space consists of real-valued functions (curves) on $\mathbb R^d$ (or some subspace thereof), but we only have access to 
discrete observations of the curves. We may have data for  $Y_{t_i}(x_j)$ at locations $x_j, j=1,\ldots, m$, or possibly some aggregation of these (or, in more generality, a finite set of linear functionals of $Y_{t_i}$). 
For example, in commodity forward markets, we have only a finite number of forward contracts traded at all times, or, like in power forward markets, we have contracts with a delivery period (see e.g. \cite{BSBK}) and hence observations of the average of $Y_{t_i}$ over intervals on $\mathbb R_+$. In other applications, like observations of temperature and wind fields in space and time, we may have accessible measurements at geographical locations where meteorological stations are situated, or, from atmospheric reanalysis where we have observations in grid cells regularly distributed in space.  From such discrete observations, one must recover the Hilbert-space elements $Y_{t_i}$. This is a fundamental issue in 
functional data analysis, and several smoothing techniques have been suggested and studied. We refer to \cite{Ramsay2005} for an extensive discussion of this. However, smoothing introduces another layer of approximation, as we do not recover $Y_{t_i}$ but some approximate version $Y^m_{t_i}$, where the superscript $m$ indicates that we have smoothed based on the $m$ available observations. The construction of a curve from discrete observations is not a unique operation as this is an inverse problem. 
In future research, it will be interesting to extend our theory to the case when (spatial) smoothing has been applied to the discrete observations.
%By suitable choice of smoothing, one would like that $Y^m_{t_i}\rightarrow Y_{t_i}$ in some suitable sense as the number of observations tends to infinity. We leave the further analysis of this to future research.   

%Another crucial result that is not covered by our analysis is the derivation of a central limit theorem for the estimator of the quadratic covariation.
%Yet, for the best of our knowledge, this paper is the only one considering high-frequency estimation for volatility of infinite dimensional stochastic evolution equations in an operator setting and therefore should be considered as a first step towards this ansatz. The methods developed in \cite{Chong2020} in his framework of parabolic equations are very powerful and allow for a central limit theorem. In future research we aim at establishing  a central limit theorem in our general framework. {\bf is it wise to say this?}

Interestingly, when we compare our work to recent developments on high-fre\-quen\-cy estimation for volatility modulated Gaussian processes in  finite dimensions, see e.g.~\cite{Podolskij2014} for a survey,  it appears that a scaling factor is needed in the realised (co)variation so that an asymptotic theory for Volterra processes can be derived. This scaling factor is given by the
variogram of the associated so-called Gaussian core process, and depends on the corresponding kernel function. 
However, in our case, due to the semigroup property, we are in a better situation than for general Volterra equations, since we actually have (or can reconstruct) the data in order to compute the semigroup-adjusted increments. We can then develop our analysis based on extending the techniques and ideas that are used in the semimartingale case. In this way, the estimator becomes independent of further assumptions on the remaining parameters of the equation.
However, the price to pay for this universality is that the convergence speed cannot generally be determined.
The semigroup-adjustment of the increments effectively forces the estimator to converge at most at the same rate as the semigroup converges to the identity on the range of the volatility as $t$ goes to $0$. 
At first glance, it seems that the strong continuity of the semigroup suggests that we can obtain convergence just with respect to the strong topology. This would make it significantly harder to apply methods from functional data analysis, even for constant volatility processes.
Fortunately, the compactness of the operators $\sigma_tQ^{\frac 12}$ for $t\in[0,T]$ comes to the rescue and enables us to prove that convergence holds with respect to the Hilbert-Schmidt norm. In this case, we obtain reasonable convergence rates for the estimator.

\appendix

\section{Proofs of some technical results}
\label{App:proofs}

\begin{proof}[Proof of Lemma \ref{lem:HS-banachalg}]
It is well-known that $L_{\text{HS}}(U,H)$ is a separable Hilbert space (see e.g. 
\citet[Appendix A.2, p. 356]{PZ2007}). Indeed, an
ONB is $(e_i\otimes f_j)_{i,j\in\mathbb N}$
where $(e_i)_{i\in\mathbb N}$ is ONB for $U$ and $(f_j)_{j\in\mathbb N}$ for $H$. 

Notice that for any 
$x\in U$, we have for $L\in L_{\text{HS}}(U,H)$
\begin{align*}
    \Vert Lx \Vert_H^2
    =\sum_{i=1}^{\infty}\langle Lx,e_i\rangle^2_H
    =\sum_{i=1}^{\infty}\langle x,L^*e_i\rangle^2_H 
    \leq \Vert x\Vert_H^2\sum_{i=1}^{\infty}\Vert L^*e_i\Vert_H^2
    =\Vert x\Vert_U^2\Vert L^*\Vert^2_{\text{HS}},
\end{align*}
where $(e_i)_{i=1}^{\infty}$ is an ONB in $U$ and we applied the Cauchy-Schwarz inequality.
Hence, $\Vert L\Vert_{\text{op}}\leq \Vert L^*\Vert_{\text{HS}}=\Vert L\Vert_{\text{HS}}$. 
It can be seen directly from definition of the Hilbert-Schmidt norm that for $L\in L_{\text{HS}}(V,H),K\in L_{\text{HS}}(U,V)$, it holds 
$$
\Vert LK\Vert_{\text{HS}}\leq \Vert L\Vert_{\text{op}}\Vert K\Vert_{\text{HS}}\leq 
\Vert L\Vert_{\text{HS}}\Vert K\Vert_{\text{HS}},
$$
and the claimed algebraic structure of Hilbert-Schmidt operators follows. 
\end{proof}

\begin{proof}[Proof of Lemma \ref{lemma:4thmoment}]
As $Q$ is a symmetric positive definite trace class operator, there exists an ONB $(e_k)_{k=1}^{\infty}$ in $U$ being the eigenvectors of $Q$. Further, recall by Fernique's Theorem (see e.g. \cite[Thm. 3.31]{PZ2007} that all moments of $\Vert\Delta W_t\Vert_U$ exists (in fact, exponential moments are finite up to a certain degree). 

By Parseval's identity,
$$
\Vert\Delta W_t\Vert_U^{2q}=\left(\sum_{k=1}^{\infty}\langle\Delta W_t,e_k\rangle_U^2\right)^q
$$
Obviously, $\sum_{k=1}^m\langle\Delta W_t,e_k\rangle_U^2$ is increasing in $m$, and it follows from Tonelli's Theorem that
$$
\mathbb E[\Vert\Delta W_t\Vert_U^{2q}]=\lim_{m\rightarrow\infty}\mathbb E\left[\left(\sum_{k=1}^{\infty}\langle\Delta W_t,e_k\rangle_U^2\right)^q\right]
$$
We find the characteristic function of $\sum_{k=1}^m\langle\Delta W_t,e_k\rangle_U^2$: By independence of the sequence of random variables $(\langle\Delta W_t,e_k\rangle_U)_{k\in\N}$ and the fact that $\langle\Delta W_t,e_k\rangle_U\sim\mathcal N(0,\Delta\lambda_k)$, it
follows for $x\in\R$
\begin{align*}
    \mathbb E\left[e^{ix\sum_{k=1}^{\infty}\langle\Delta W_t,e_k\rangle_U^2}\right]&=\times_{k=1}^m\mathbb E\left[e^{ix\langle\Delta W_t,e_k\rangle_U^2}\right]=\times_{k=1}^m\mathbb E\left[e^{ix\Delta\lambda_k Z^2}\right],
\end{align*}
with $Z\sim\mathcal N(0,1)$. But $Z^2$ is $\chi^2(1)$-distributed, and thus,
\begin{align*}
    \mathbb E\left[e^{ix\sum_{k=1}^m\langle\Delta W_t,e_k\rangle_U^2}\right]&=\prod_{k=1}^m\left(1-2ix\Delta\lambda_k\right)^{-1/2} \\
    &=\prod_{k=1}^m\exp\left(-\frac12\ln(1-2ix\Delta\lambda_k)\right) \\
    &=\exp\left(-\frac12\sum_{k=1}^m\ln(1-2ix\Delta\lambda_k)\right).
\end{align*}
As the $q$th moment of a random variable multiplied by $i^q$ is given by the $q$th derivative of its characteristic function evaluated at zero, the first part of the Lemma follows. 

For the second part, we first find that
$$
\Phi''_m(x)=-\Delta^2\Phi(x)\left(\left(\sum_{k=1}^m\frac{\lambda_k}{1-2ix\Delta\lambda_k}\right)^2+2\sum_{k=1}^m\frac{\lambda_k^2}{(1-2ix\Delta\lambda_k)^2}\right)
$$
and therefore,
$$
\mathbb E\left[\Vert\Delta W_t\Vert_U^4\right]=(-i)^2\lim_{m\rightarrow\infty}\Phi_m''(0)=\Delta^2\left(\sum_{k=1}^{\infty}\lambda_k\right)^2+2\Delta^2\sum_{k=1}^{\infty}\lambda_k^2.
$$
The result follows. 
%We denote the eigenvalues $(\lambda_i)_{i=1}^{\infty}\subset\mathbb R_+$, where $\text{Tr}(Q)=\sum_{i=1}^{\infty}\lambda_i<\infty$. From Parseval's identity, we get
%\begin{align*}
%    \Vert\Delta W_t\Vert^4_U&=\left(\sum_{i=1}^{\infty}\langle\Delta W_t,e_i\rangle_U^2\right)^2 \\
%    &=\sum_{i,j=1, i\neq j}^{\infty}\langle\Delta W_t,e_i\rangle_U^2\langle\Delta W_t,e_j\rangle_U^2+\sum_{i=1}^{\infty}\langle\Delta W_t,e_i\rangle_U^4.
%%\end{align*}
%Notice that $\langle\Delta W_t,e_i\rangle_U$ and $\langle\Delta W_t,e_j\rangle_U$ are
%independent Gaussian random variables for $i\neq j$ since $\langle\Delta W_t,e_i\rangle_U$ is
%Gaussian and 
%$$
%\mathbb{E}[\langle\Delta W_t,e_i\rangle_U\langle\Delta W_t,e_j\rangle_U]=\langle Q e_i,e_j\rangle_U\Delta=\lambda_i\langle e_i,e_j\rangle_U\Delta=0.
%%$$
%Furthermore,
%$$
%\mathbb{E}[\langle\Delta W_t,e_i\rangle_U^2]=\langle Qe_i,e_i\rangle_U\Delta=\lambda_i\Delta.
%$$
%Therefore, in distribution, we find $\langle\Delta W_t,e_i\rangle_U=\sqrt{\lambda_i\Delta}\,X$ for some $X\sim N(0,1)$, and 
%$$
%\mathbb{E}[\langle\Delta W_t,e_i\rangle_U^4]=\lambda_i^2\Delta^2\mathbb E[X^4]=3\lambda_i^2\Delta^2.
%$$
%Hence, we calculate,
%\begin{align*}
%    \mathbb{E}[\Vert\Delta W_t\Vert^4_U]&=\sum_{i,j=1, i\neq j}^{\infty}\mathbb{E}[\langle\Delta W_t,e_i\rangle_U^2]\mathbb{E}[\langle\Delta W_t,e_j\rangle_U^2] \\
%    &\qquad+\sum_{i=1}^{\infty}\mathbb{E}[\langle\Delta W_t,e_i\rangle_U^4] \\
%    &=\Delta^2\left(\sum_{i,j}^{\infty}\lambda_i\lambda_j+2\sum_{i=1}^{\infty}\lambda_i^2\right),
%\end{align*}
%which shows the result.
\end{proof}

\bibliographystyle{agsm}
\bibliography{Bib_HilbertVol}
\end{document}